\documentclass[11pt,reqno]{amsart}
\usepackage{amsfonts,amssymb,amsmath,amsthm}
\usepackage[margin=1in]{geometry}
\usepackage[hidelinks]{hyperref}
\usepackage{enumerate}
\usepackage{microtype}
\usepackage{mathtools}
\usepackage{makecell}
\usepackage{multirow}
\usepackage{cite}
\usepackage{color}
\usepackage{tikz-cd}
\usepackage{float}
\usepackage{longtable}
\usepackage{multicol}

\newcommand\Z{\mathbb{Z}}
\newcommand\Q{\mathbb{Q}}
\newcommand\F{\mathbb{F}}

\newcommand{\mc}[1]{\href{https://beta.lmfdb.org/ModularCurve/Q/#1/}{\texttt{#1}}}
\newcommand{\ec}[1]{\href{https://www.lmfdb.org/EllipticCurve/Q/#1}{\texttt{#1}}}

\DeclareMathOperator{\Gal}{Gal}
\DeclareMathOperator{\GL}{GL}
\DeclareMathOperator{\SL}{SL}
\DeclareMathOperator{\PGL}{PGL}

\DeclareMathOperator{\Aut}{Aut}
\DeclareMathOperator{\rad}{rad}
\DeclareMathOperator{\ord}{ord}
\DeclareMathOperator{\lcm}{lcm}
\DeclareMathOperator{\Quo}{Quo}
\DeclareMathOperator{\tors}{tors}

\theoremstyle{plain}
\newtheorem{theorem}{Theorem}
\newtheorem{corollary}[theorem]{Corollary}
\newtheorem{lemma}[theorem]{Lemma}
\newtheorem{proposition}[theorem]{Proposition}

\theoremstyle{definition}

\newtheorem{conjecture}[theorem]{Conjecture}

\theoremstyle{remark}
\newtheorem{remark}[theorem]{Remark}

\title[Possible adelic indices for elliptic curves admitting a rational cyclic isogeny]{The possible adelic indices for elliptic curves admitting a rational cyclic isogeny}

\author[Finnerty]{Kate Finnerty}
\address{Kate Finnerty, Department of Mathematics and Statistics, Boston University, Boston, MA 02215}
\email{ksfinn@bu.edu}
\author[Genao]{Tyler Genao}
\address{Tyler Genao, Department of Mathematics, The Ohio State University, Columbus, OH 43210}
\email{genao.5@osu.edu}
\author[Mayle]{Jacob Mayle}
\address{Jacob Mayle, Department of Mathematical Sciences, University of Delaware, Newark, DE 19716}
\email{mayle@ud.edu}
\author[Rakvi]{Rakvi}
\address{Rakvi, Department of Mathematics, The University of Maine, Orono, ME 04469} 
\email{raakvi@gmail.com}

\date{\today}

\begin{document}
\begin{abstract} In the 1970s, Serre proved that the adelic index of a non-CM elliptic curve over a number field is finite. More recently, Zywina conjectured the complete set of adelic indices for such curves over $\Q$. In this article, we prove that Zywina's conjecture is true for the family of non-CM elliptic curves over $\Q$ that admit a nontrivial rational cyclic isogeny. This strengthens a result of Lemos that resolved Serre's uniformity question for the same family of curves. Our proof proceeds by analyzing a collection of modular curves associated with each prime isogeny degree, using recent advances on $\ell$-adic images, isogeny-torsion graphs, and computations of models and rational points.
\end{abstract}
\maketitle

\section{Introduction}

Let $E$ be an elliptic curve defined over a number field $F$. For each positive integer $n$, the natural action of the absolute Galois group $\Gal(\overline{F}/F)$ on the $n$-torsion subgroup $E[n] \cong (\Z/n\Z)^2$ gives rise to the \emph{mod $n$ Galois representation of $E$}
\[
\rho_{E,n}\colon \Gal(\overline{F}/F)\longrightarrow \Aut(E[n])\cong \GL_2(\Z/n\Z).
\]
Taking the inverse limit of $E[n]$ over integers $n \geq 1$ ordered by divisibility (with a compatible choice of bases for $E[n]$), we obtain the \emph{adelic Tate module}
\( T(E)\coloneqq\varprojlim E[n],\)
which is a free module of rank $2$ over the ring $\widehat{\Z}$ of profinite integers. The action of $\Gal(\overline{F}/F)$ on each $E[n]$ extends to a continuous action on $T(E)$, giving the \emph{adelic Galois representation of $E$}
\[
\rho_E\colon \Gal(\overline{F}/F)\longrightarrow \Aut(T(E))\cong \GL_2(\widehat{\Z}).
\]
We denote the \emph{adelic image of $E$} by $G_E\coloneqq\rho_E(\Gal(\overline{F}/F))$. This is well-defined up to conjugacy, because choosing an isomorphism $\Aut(T(E))\cong \GL_2(\widehat{\Z})$ amounts to choosing a basis (and likewise for the isomorphism $\Aut(E[n])\cong \GL_2(\Z/n\Z)$). For each positive integer $n$, the image of $G_E$ under the mod $n$ reduction map $\pi_n\colon \GL_2(\widehat{\Z})\rightarrow \GL_2(\Z/n\Z)$, written as $G_E(n)$, is equal to $\rho_{E,n}(\Gal(\overline{F}/F))$. Throughout this article, the base field $F$ will be clear from context (and will almost always be $\Q$).

The adelic Galois representation of an elliptic curve encodes all information about the rationality of its torsion points. It is therefore natural to ask what can be said about the adelic image $G_E$. As it turns out, the geometric endomorphism ring $\text{End}(E)\coloneqq\text{End}_{\overline{F}}(E)$ of $E$ has a strong influence on the structure of $G_E$. An elliptic curve $E$ is \emph{non-CM} (or \emph{without complex multiplication}) if $\text{End}(E)$ is as small as possible, i.e., $\text{End}(E)=\Z$.
A landmark result of Serre states that the adelic image of $E$ is ``large'' when $E$ is non-CM. More precisely, we have the following.

\begin{theorem}[Serre's Open Image Theorem, \cite{Ser72}]\label{Thm_SerresOIT}

Let $E$ be an elliptic curve over a number field. If $E$ is non-CM, then $G_E$ is an open subgroup of $\GL_2(\widehat{\Z})$, equivalently of finite index in $\GL_2(\widehat{\Z})$.
\end{theorem}
Serre's Open Image Theorem is equivalent to the existence of a positive integer $n$ such that $G_E = \pi_n^{-1}(G_E(n))$; the least such $n$ is called the \emph{adelic level} of $E$. Since Serre proved Theorem \ref{Thm_SerresOIT}, there has been substantial progress  toward understanding Galois images of non-CM elliptic curves over $\Q$. Notably, Zywina \cite{zywina2015possibleimagesmodell} and Sutherland \cite{MR3482279}  conjecturally classified all possible mod $\ell$ Galois images and gave efficient algorithms for computing them. Rouse--Zureick-Brown \cite{MR3500996} and Rouse--Sutherland--Zureick-Brown \cite{MR4468989} did similarly for $\ell$-adic images. In 2022, Zywina \cite{ZywinaOpen} gave a breakthrough algorithm for computing adelic images of elliptic curves.

An important invariant related to the adelic Galois image of a non-CM elliptic curve $E$ is the \emph{adelic index}, the quantity $[\GL_2(\widehat{\Z}):G_E]$. In general, the adelic index cannot be determined from the $\ell$-adic indices alone because of possible entanglements between the $\ell$-adic images. For example, consider the elliptic curve \ec{37.a1} in the LMFDB. For all primes $\ell$, its $\ell$-adic image is $\GL_2(\Z_{\ell})$, yet its adelic index is 2. In fact, the adelic index of any non-CM elliptic curve over $\Q$ must be at least $2$, as noted in \cite[Proposition 22]{Ser72}. Moreover, Jones \cite{MR2563740} proved that, when ordered by naive height, almost all elliptic curves over $\Q$ have adelic index $2$. Adelic indices greater than 2 occur for infinitely many $\overline{\Q}$-isomorphism classes \cite[Theorem 4]{MR2563740}, but these are sparse asymptotically.

Zywina conjectured that there are only finitely many possible adelic indices for non-CM elliptic curves over $\Q$, and he even proposed a specific finite set $\mathcal{I}$ of all  adelic indices \cite{zywina2022possibleindicesgaloisimage, ZywinaOpen}.

\begin{conjecture}[Conjecture 1.5, \cite{ZywinaOpen}]\label{Conj_PossibleAdelicIndices}

If $E/\Q$ is a non-CM elliptic curve, then 
\[[\GL_2(\widehat{\Z}) : G_E] \in \mathcal{I},\]
where $\mathcal{I}$ is the finite set
\[
\mathcal{I}\coloneqq \left\{\begin{array}{c} 2, 4, 6, 8, 10, 12, 16, 20, 24, 30, 32, 36, 40, 48, 54, 60, 72, 80, 84, 96, 108,\\ 112, 120, 128, 144, 160, 182, 192, 200, 216, 220, 224, 240, 288, 300, 336, \\360, 384, 480, 504, 576, 768, 864, 1152, 1200, 1296, 1536, 2736 \end{array}\right\}. 
\]
\end{conjecture}

This conjecture has been verified for all non-CM elliptic curves over $\Q$ currently appearing in the $L$-Functions and Modular Forms Database (LMFDB) \cite{lmfdb}, consisting of more than three million curves. The conjecture is closely connected to {Serre's uniformity question,} which asks whether there exists a prime $p$ such that for any prime $\ell>p$, the mod $\ell$ Galois representation of any non-CM elliptic curve $E / \Q$ is surjective. It is now widely conjectured that Serre's uniformity question has an affirmative answer with $p=37$. 

Over the past several decades, significant progress toward Serre's uniformity question has come from the study of rational points on modular curves along Dickson's classification: if the mod $\ell$ representation is not surjective, its image is contained in a Borel subgroup, the normalizer of a split or nonsplit Cartan subgroup, or an exceptional subgroup. Serre \cite{Ser72} proved that the exceptional case cannot occur for $\ell > 13$. Mazur \cite{MR482230} proved that if $\ell>37$, then $G_E(\ell)$ is not contained in a Borel subgroup. Bilu and Parent \cite{MR2753610} proved that $G_E(\ell)$ cannot be contained in the normalizer of a split Cartan subgroup for sufficiently large $\ell$.  Bilu, Parent, and Rebolledo \cite{MR3137477} later showed that the same result holds for \emph{every} prime $\ell \geq 11$, except  possibly $\ell=13$. The latter case was subsequently resolved via quadratic Chabauty by Balakrishnan, Dogra, M\"{u}ller, Tuitman, and Vonk \cite{MR3961086}. As for the case of nonsplit Cartan subgroups, recent work of Furio and Lombardo \cite{FurioLombardo23} shows that if $\ell>37$ is nonsurjective, then $G_E(\ell)$ is \emph{precisely equal to} the normalizer of a nonsplit Cartan subgroup. Showing that this does not happen for sufficiently large $\ell$ is the remaining problem for the uniformity question, but eliminating this case appears daunting.

Nevertheless, Serre's uniformity question is known to have a positive answer for certain families of non-CM elliptic curves over $\Q$ due to work of Lemos. He proved that  uniformity holds with $p = 37$ for the family of non-CM elliptic curves $E/\Q$ for which there exists a prime $q$ such that $G_E(q)$ is contained in the normalizer of a split Cartan subgroup of $\GL_2(\F_q)$ \cite{MR3968924}. Further, he proved that uniformity holds with $p = 37$ for the family of all non-CM elliptic curves admitting a nontrivial rational cyclic isogeny, which will be a starting point for this article.

\begin{theorem}[Theorem 1.1, \cite{MR3885140}]\label{Thm_Lemos}
Let $E/\Q$ be a non-CM elliptic curve that admits a nontrivial rational cyclic isogeny. Then for all primes $\ell>37$, the mod $\ell$ Galois representation of $E$ is surjective.
\end{theorem}

In this article, we prove Zywina's conjecture for all non-CM elliptic curves over $\Q$ that admit a nontrivial rational cyclic isogeny. Our main result, stated below, both builds upon and strengthens  Theorem \ref{Thm_Lemos}.

\begin{theorem}\label{t:mainthm} Let $E/\Q$ be a non-CM elliptic curve and $p$ be a prime number. If $E$ admits a rational isogeny of degree $p$, then  
\[ [\GL_2(\widehat{\Z}) : G_E ] \in \mathcal{I}_p, \]
where
\begin{align*}
    \mathcal{I}_2 &\coloneqq \{12, 48, 72, 96, 144, 192, 288, 384, 576, 768, 864, 1152, 1536\}, \\
    \mathcal{I}_3 &\coloneqq \{16, 32, 48, 96, 128, 144, 160, 288, 384, 768, 864, 1296 \}, \\
    \mathcal{I}_5 &\coloneqq \{48,192,288,384,576,864,1200\},\\  
    \mathcal{I}_7 &\coloneqq \{96,192,576,768\}, \\
    \mathcal{I}_{11} &\coloneqq \{480\},\\
    \mathcal{I}_{13} &\coloneqq \{336\},\\
    \mathcal{I}_{17} &\coloneqq \{576\}, \\
    \mathcal{I}_{37} &\coloneqq \{2736\}.
\end{align*}
In particular, Conjecture \ref{Conj_PossibleAdelicIndices} holds for all non-CM elliptic curves over $\Q$ that admit a nontrivial rational cyclic isogeny.
\end{theorem}  

The theorem above is sharp in the sense that for each prime $p \in \{2,3,5,7,11,13,17,37\}$ and each integer $n \in \mathcal{I}_p$, there exists a non-CM elliptic curve over $\Q$ admitting a rational isogeny of degree $p$ with index $n$. Such examples can be found using the search functionality of the LMFDB and are recorded in Appendix \ref{A:ExCurves}. These values of $p$ are precisely the possible degrees of a prime degree rational isogeny for a non-CM elliptic curve over $\Q$ by Mazur's theorem \cite{MR482230}. The final sentence of the theorem follows as a consequence of the results for individual prime degrees. Indeed, if $E/\Q$ admits a rational cyclic isogeny of degree $N \ge 2$, then $E$ admits a rational isogeny of degree $p$ for each prime $p$ dividing $N$, so its adelic index belongs to $\mathcal{I}_p \subseteq \mathcal{I}$ for each such $p$. 

The proof of Theorem \ref{t:mainthm} extends throughout most of this paper. Section \ref{s:prelim} reviews background on Galois representations, isogenies, twists, and modular curves. Section \ref{s:rationalpoints} discusses methods for computing rational points on modular curves. In Section \ref{s:possibleimages}, we determine all possible $\ell$-adic images containing $-I$ for non-CM elliptic curves admitting a rational isogeny of degree $p$ (and no rational isogeny of a larger prime degree) by studying rational points on modular curves and using Mazur and Kenku's results on isogenies \cite{MR482230,MR675184} along with Lemos's theorem (Theorem \ref{Thm_Lemos}). In Section \ref{s:adelicindex}, we use a result of Jones (Theorem \ref{P:m0}) to give a product formula for the adelic index involving $\ell$-adic indices for certain small primes and an entanglement factor. 

We complete the proof of Theorem \ref{t:mainthm} in Sections \ref{s:11and17and37}, \ref{s:5and7and13}, and \ref{s:2and3}. In Section \ref{s:11and17and37}, we consider the primes $p \in \{11,17,37\}$, for which the proof is straightforward. In general, the adelic index of an elliptic curve over $\Q$ depends only on its $j$-invariant (Lemma \ref{lem:jinv}), and in these cases, the modular curve $X_0(p)$ has only finitely many rational points. Thus, it suffices to compute the adelic index for one representative elliptic curve for each non-CM rational point on $X_0(p)$. In Section \ref{s:5and7and13}, we treat the primes $p \in \{5, 7, 13\}$ using a lattice-based approach that builds on ideas from \cite{MR3500996,MR4468989,ZywinaOpen}. In principle, the same approach could be applied for $p \in \{2,3\}$, but it would be computationally impractical due to the number of groups involved. Instead, in Section \ref{s:2and3} we take a product-based approach to handle the primes $p \in \{2,3\}$.

This article makes extensive use of computations carried out in \texttt{Magma}, Version \texttt{V2.28-21} \cite{MR1484478}. The GitHub repository containing code for this article is:

\centerline{\url{https://github.com/Rakvi6893/Adelic-indices-of-elliptic-curves-with-isogeny/}}

\noindent We make extensive use of data from the LMFDB \cite{lmfdb}. Throughout the article, we use the LMFDB labeling conventions for elliptic curves and modular curves. For elliptic curves, the label \texttt{N.aX} specifies the conductor, followed by the isogeny class index and the isomorphism class index within the isogeny class. For modular curves, the label \texttt{N.i.g.c.n} specifies the level, index, genus, Gassmann class identifier, and an integer that distinguishes nonconjugate subgroups when the first four identifiers coincide. The modular curves section of the LMFDB is currently in beta. All LMFDB labels are hyperlinked to the object's homepage. 

\subsection{Acknowledgments} We would like to thank Eran Assaf and Jeremy Rouse for their helpful suggestions related to computations of rational points on modular curves. We also thank Andrew Sutherland, Jeremy Rouse, and Jennifer Balakrishnan for their feedback on an earlier version of this article. The first-named author was partially supported by US-Israel BSF Grant 2022393.

\section{Preliminaries}\label{s:prelim}

\subsection{Galois Representations}

In the introduction, we discussed the mod $n$ and adelic Galois representations of an elliptic curve $E$ over a number field $F$. Between these, one can also consider the $n$-adic Galois representation of $E$. For any positive integer $n$, the ring of $n$-adic integers is
\[
\Z_n\coloneqq\varprojlim \Z/n^k\Z \cong \prod_{\ell\mid n}\Z_\ell.
\]
The natural action of $\Gal(\overline{F}/F)$ on the $n$-adic Tate module $T_n(E)=\varprojlim E[n^k]$ gives rise to the \emph{$n$-adic Galois representation}
\[
\rho_{E,n^\infty}\colon \Gal(\overline{F}/F)\longrightarrow \GL_2(\Z_n).
\] 
We may view $\rho_{E,n^\infty}$ as the composition of the adelic Galois representation $\rho_E \colon \Gal(\overline{F}/F) \to \GL_2(\widehat{\Z})$ with the $n$-adic projection map $\pi_{n^\infty}\colon \GL_2(\widehat{\Z})\rightarrow \GL_2(\Z_n)$. We denote the image of $\rho_{E,n^\infty}$ by $G_E(n^\infty)$.

Each of the Galois images $G_E, G_E(n^\infty),$ and $G_E(n)$ depends on the chosen basis and thus is well defined only up to conjugation. When defining the Galois representations, we adopt the convention that torsion points and Tate modules are represented as column vectors and that all matrix groups act by left multiplication. This agrees with the convention currently used in the LMFDB but differs from the convention used by Zywina in \cite{ZywinaOpen, ZywinaModular}. Consequently, when using Zywina's code to compute a Galois image or a modular curve, we must take a transpose.

For a subgroup $G\subseteq \GL_2(\widehat{\Z})$, the \emph{level} of $G$ is defined to be the least positive integer $n$ such that
\[ G = \pi_n^{-1}(\pi_n(G)), \]
provided such an $n$ exists. If $G$ is open with respect to the profinite topology of $\GL_2(\widehat{\Z})$, then such an $n$ must exist. We write $\det G$ to denote the image of $G$ under the determinant map 
\[ \det\colon \GL_2(\widehat{\Z})\longrightarrow \widehat{\Z}^\times.\]
We also write $G'\coloneqq[G, G]$ to denote the commutator subgroup of $G$. 

We continue to let $E$ denote an elliptic curve over a number field $F$. Write
$\chi \colon \Gal(\overline{F} / F) \to \widehat{\Z}^\times$
to denote the adelic cyclotomic character. It follows from the Weil pairing on $E$ that
\( \det \circ \rho_{E} = \chi\); see \cite[\S 3.8]{MR2514094} and \cite[IV-18 Lemma 2]{MR263823}. 
In particular, if $F = \Q$, then
\[
\det G_E = \chi(\Gal(\overline{\Q}/\Q)) = \widehat{\Z}^\times.
\]
In other words, the composition $\det \circ \rho_E$ is surjective. Lastly, we note that in the case  $F = \Q$, if $\ell \geq 5$ is a prime, then $\rho_{E,\ell}$ is surjective if and only if $\rho_{E,\ell^\infty}$ is surjective \cite[IV-23 Lemma 3]{MR263823}.

\subsection{Isogenies} \label{Sec:Isog}
In this subsection, we review some standard facts about isogenies and prove a lemma about the Galois images of elliptic curves related by a rational prime degree isogeny. 

First, let $\phi:E\rightarrow E'$ be an isogeny of two elliptic curves over a number field $F$. We call $\phi$ an \emph{$n$-isogeny} if its kernel $C$ is a cyclic group of order $n$. We say that $\phi$ is \emph{$F$-rational} if $C$ is stable under the action of $\Gal(\overline{F}/F)$, i.e., for all $\sigma\in \Gal(\overline{F}/F)$ and all $R\in C$, one has $\sigma(R)\in C$. Any cyclic subgroup $C\subseteq E(\overline{F})$ of order $n$ stabilized by $\Gal(\overline{F}/F)$ induces an $F$-rational $n$-isogeny to some elliptic curve $E'/F$ whose kernel is $C$; see \cite[Exercise 3.13]{MR2514094}. If $F=\Q$, then instead of saying \emph{$\Q$-rational}, we simply say \emph{rational}.

A well-known result of Kenku extends Mazur's \cite{MR482230} classification of rational isogenies of prime degree to all degrees $n$ for which a rational $n$-isogeny exists.

\begin{theorem}[Theorem 1, {\cite{MR675184}}]\label{thm:isogeniescomposite}
 Let $E/\Q$ be a non-CM elliptic curve. Then $E$ admits a rational $n$-isogeny if and only if $n \in \{1,2,3,4,5,6,7,8,9,10,11,12,13,15,16,17,18,21,25,37\}$.
\end{theorem}

In general, for an elliptic curve $E/F$ and a positive integer $n$, one has that $E$ admits an $F$-rational $n$-isogeny if and only if $G_E(n)$ is conjugate to a subgroup of the Borel subgroup:
\[B_0(n)
\coloneqq\left\lbrace \begin{bmatrix}
a&b\\
0&d
\end{bmatrix}\in \GL_2(\Z/n\Z) : a,d \in (\Z/n\Z)^{\times}, b \in  \Z/n\Z \right\rbrace. \]

For a prime $p$, there is a $p$-adic variant of $B_0(n)$ which will prove useful. Let us define the group
\[
B_0(p^\infty)\coloneqq\left\{
\begin{bmatrix}
    \alpha & \beta \\ \gamma & \delta
\end{bmatrix}\in \GL_2(\Z_p) :\alpha,\delta \in \Z_p^\times, \beta \in \Z_p, \text{ and } \gamma \in p \Z_p
\right\}.
\]
Observe that $B_0(p^\infty)$ is the full inverse image of the Borel subgroup $B_0(p)$ under the reduction map $\GL_2(\Z_p)\rightarrow \GL_2(\Z/p\Z)$. From this description and our above discussion, one obtains the following.

\begin{lemma}
    Let $p$ be a prime number and $E/F$ be an elliptic curve. Then $E$ admits a rational $p$-isogeny if and only if $G_E(p^\infty)$ is conjugate to a subgroup of $B_0(p^\infty)$.
\end{lemma}

We now state and prove a lemma relating the mod $n$ Galois images for isogenous elliptic curves (cf.\ \cite[Corollary 2.21]{MR4868206}). 

\begin{lemma}\label{lemma:isomorphicimages}
   Let $E_1/F$ and $E_2/F$ be elliptic curves such that for some prime $p$ there exists an $F$-rational $p$-isogeny $\phi\colon E_1\rightarrow E_2$. Then for all positive integers $n$ coprime to $p$, the groups $G_{E_1}(n)$ and $G_{E_2}(n)$ are conjugate.
\end{lemma}

\begin{proof}
    Let $n$ be coprime to $p$. The isogeny $\phi\colon E_1\rightarrow E_2$ restricts to a homomorphism 
    \[ \phi_n \colon E_1[n] \longrightarrow E_2[n]\]
    whose kernel is contained in the kernel of $\phi$. Since $n$ is coprime to $p$, we have $E_1[n]\cap \ker\phi= \{\mathcal{O}_{E_1}\}$, so $\phi_n$ is injective. As $|E_1[n]|=|E_2[n]|=n^2$, it follows that $\phi_n$ is an isomorphism. Because $\phi$ is $F$-rational, we conclude that $\phi_n$ is in fact a $\Gal(\overline{F}/F)$-module isomorphism. 
   
   Next let $\{P,Q\}$ be a basis for $E_1[n]$. Then $\{\phi(P), \phi(Q)\}$ is a basis for $E_2[n]$. Since $\phi_n$ is a $\Gal(\overline{F}/F)$-module isomorphism, we have that for any $\sigma \in \Gal(\overline{F}/F)$,
    \begin{align*}
        \sigma(P) &= aP+cQ \\
        \sigma(Q) &= bP+dQ
\end{align*} if and only if \begin{align*}
        \sigma(\phi(P))&=a\phi(P)+c\phi(Q) \\
        \sigma(\phi(Q))&=b\phi(P)+d\phi(Q).
    \end{align*} 
    In other words,
    \( \rho_{E_1,n}(\sigma) = \rho_{E_2,n}(\sigma) \)
    where $\rho_{E_1,n}$ is with respect to the basis $\{P,Q\}$ and $\rho_{E_2,n}$ is with respect to the basis $\{\phi(P),\phi(Q)\}$. In particular, $G_{E_1}(n)$ is conjugate to $G_{E_2}(n)$.
\end{proof}

\subsection{Twists}
In this subsection, we collect several standard facts about how Galois representations behave under quadratic twisting. We will use these results later to simplify our computations. Although we state these results over $\Q$, some of them hold over general number fields as well.

First, we recall that for non-CM elliptic curves over $\Q$, the adelic index depends only on the $j$-invariant, which characterizes the $\overline{\Q}$-isomorphism class of the curve. 

\begin{lemma}[Corollary 2.3, \cite{zywina2022possibleindicesgaloisimage}]\label{lem:jinv}
    Let $E/\Q$ be a non-CM elliptic curve. Then the adelic index of $E$ depends only on the $j$-invariant of $E$.
\end{lemma}

It is worth noting that if $E$ and $E'$ are non-CM elliptic curves over $\Q$ with the same $j$-invariant, then $E$ and $E'$ are isomorphic over some quadratic field; in other words, they are quadratic twists of one another. Next, we recall that the existence of a rational $n$-isogeny is invariant under quadratic twisting.

\begin{lemma}[Proposition 3.5, \cite{MR3704357}]\label{lem:twistisog}
    Let $E/\Q$ be an elliptic curve that admits a rational $n$-isogeny for some positive integer $n$. Then every quadratic twist $E'/\Q$ of $E$  also admits a rational $n$-isogeny.
\end{lemma}

Lastly, we record how the mod $n$ Galois representations of an elliptic curve and its quadratic twists are related. If $E^d$ denotes the quadratic twist of $E$ by $d$, then up to an appropriate choice of bases,
\[
\rho_{E^d,n}(\sigma)=(\chi_d\cdot \rho_{E,n})(\sigma),
\]
 where $\chi_d \colon \Gal(\overline{\Q} / \Q) \to \{\pm 1\}$ is the quadratic character corresponding to $\Q(\sqrt{d})$. An important consequence of this relation is the following lemma, which implies that for any integer $n > 2$ and any elliptic curve $E/\Q$, one can find a quadratic twist $E'$ of $E$ such that $-I\in G_{E'}(n)$. 

\begin{lemma}[Corollary 5.25(b), \cite{MR3482279}]\label{lem:-Itwist}
    Let $E/\Q$ be an elliptic curve and $n>2$ be an integer such that $-I\notin G_E(n)$. Let $E'/\Q$ be a quadratic twist of $E$. Then $G_{E'}(n)$ is conjugate to either $G\coloneqq\langle G_E(n),-I\rangle$ or an index $2$ subgroup of $G$ that does not contain $-I$. The latter occurs precisely when the quadratic extension corresponding to the twist is a subfield of the splitting field of the $n$-division polynomial of $E$. In particular, there always exists a twist $E' / \Q$ such that  $-I \in G_{E'}(n)$. 
\end{lemma}

\subsection{Level Results}\label{s:levelresults}

When carrying out computations involving adelic images, it is necessary to work at a sufficiently large finite level. In this subsection, we recall two results that help determine an appropriate level for our computations.

Given an integer $N=\prod_{\ell}\ell^{e_\ell}$, define an integer $N_\ell$ for each prime $\ell$ dividing $N$ as follows:
\[
N_\ell \coloneqq \ell^{e_\ell}\cdot \prod_{{\substack{p \mid N \\ p^2\equiv 1 \bmod{\ell}}}} p.
\]
With this notation in place, we have the following proposition.

\begin{proposition}[Lemma 7.2, 7.4, {\cite{ZywinaOpen}}] \label{L:Zyw-Lemma7.2}
Fix an integer $N=\prod_\ell \ell^{e_\ell}>1$ with $e_2 \neq 1$. 
\begin{itemize}
\item Let $G$ be a closed subgroup of $\GL_2(\Z_N)$.   For each prime $\ell\mid N$, assume that the image of $G$ in $\GL_2(\Z/N_\ell\ell\Z)$ contains all matrices that are congruent to $I$ modulo $N_\ell$. Then $G$ is an open subgroup of $\GL_2(\Z_N)$ with level dividing $N$. 
\item Let $G$ be a closed subgroup of $\SL_2(\Z_N)$.   For each prime $\ell\mid N$, assume that the image of $G$ in $\SL_2(\Z/N_\ell\ell\Z)$ has level dividing $N_\ell$. Then $G$ is an open subgroup of $\SL_2(\Z_N)$ with level dividing $N$.
\end{itemize}
\end{proposition}

This proposition yields the following corollary about the level of maximal open subgroups.

\begin{corollary}[Lemma 7.3, {\cite{ZywinaOpen}}] \label{L:Zyw-Lemma7.3}
Fix an integer $N=\prod_\ell \ell^{e_\ell}>1$ with $e_2\neq 1$.   Let $G$ be an open subgroup of $\GL_2(\Z_N)$ whose level divides $N$.  If $M$ is a maximal open subgroup of $G$, then the level of $M$ divides $N \ell$ for some prime divisor $\ell$ of $N$.
\end{corollary}

\subsection{Modular Curves}\label{sec:modcurveintro} In this subsection, we discuss some background material on modular curves. Methods for computing rational points will be discussed in Section \ref{s:rationalpoints}.

Let $G$ be an open subgroup of $\GL_2(\widehat{\Z})$ that contains $-I$ and has full determinant, i.e., $\det(G)=\widehat{\Z}^{\times}$. Let $N$ denote the level of $G$. By working modulo $N$, we may regard $G$ as a subgroup of $\GL_2(\Z/N\Z)$ and obtain the associated modular curve $X_G$. 
We will use Zywina's code \cite{ZywinaModular} to compute models of modular curves.

We recall the following characterization of noncuspidal points. This result is well known (see \cite[Proposition 3.3]{zywina2015possibleimagesmodell}, for instance).

\begin{proposition}\label{prop:modular curves}
   Let $G \subseteq \GL_2(\widehat{\Z})$ be an open subgroup containing $-I$ and for which $\det G = \widehat{\Z}^\times$. Let $E$ be an elliptic curve defined over $\Q$ such that the $j$-invariant of $E$ is neither $0$ nor $1728$. Then $G_E$ is conjugate to a subgroup of $G$ if and only if the $j$-invariant of $E$ lies in the image of $X_G(\Q)$ under the $j$-map.
\end{proposition}

When $G$ is a Borel subgroup $B_0(N)$ of $\GL_2(\Z/N\Z)$, we denote $X_G$ by $X_0(N)$. By Proposition \ref{prop:modular curves} and our discussion in Section \ref{Sec:Isog}, the rational noncuspidal points of $X_0(N)$ correspond to elliptic curves defined over $\Q$ (up to $\overline{\Q}$ isomorphism) that admit a rational $N$-isogeny. We write $X_{\mathrm{sp}}(N)$ and $X_{\mathrm{ns}}(N)$ to denote the modular curves associated with the split and nonsplit Cartan subgroups, respectively, and write $X_{\mathrm{sp}}^+(N)$ and $X_{\mathrm{ns}}^+(N)$ for the modular curves associated with their normalizers. Finally, for $N$ prime, we write $X_{S_4}(N)$ for the modular curve associated with the inverse image in $\GL_2(\Z/N\Z)$ of the unique (up to conjugacy) subgroup of $\PGL_2(\Z/N\Z)$ that is isomorphic to $S_4$. We will often also denote these modular curves by their LMFDB labels, as discussed in the introduction. With some abuse of notation, we will refer to $G$ and $X_G$ interchangeably when it is clear from the context. 

Many of our computations with modular curves are motivated by a need to determine whether there exists a non-CM elliptic curve over $\Q$ whose adelic Galois image is contained in the intersection of two open subgroups of $\GL_2(\widehat{\Z})$ that contain $-I$ and have full determinant. To this end, we define the fiber product of two modular curves. If $G_1$ and $G_2$ are two open subgroups of $\GL_2(\widehat{\Z})$, of levels $N_1$ and $N_2$, respectively, that both contain $-I$ and have full determinant, then the \emph{fiber product} of the modular curves $X_{G_1}$ and $X_{G_2}$ is the fiber product of these curves as stacks over $X_{\GL_2(\widehat{\Z})}$. If $N_1$ and $N_2$ are coprime, then the fiber product of $X_{G_1}$ and $X_{G_2}$ is simply the modular curve $X_{G_1\cap G_2}$.

\section{Rational Point Computations}\label{s:rationalpoints}

The proof of Theorem \ref{t:mainthm} requires computing the rational points of many modular curves. These computations appear throughout Sections \ref{s:possibleimages}, \ref{s:11and17and37}, \ref{s:5and7and13}, and \ref{s:2and3}. They are sometimes straightforward (e.g., when the curve is an elliptic curve of rank $0$ or when there is a local obstruction), but often  rely on more advanced techniques. In this section, we describe the methods that we will use repeatedly.

\subsection{Map to a Rank 0 Elliptic Curve} \label{S:M2EC}
Sometimes a modular curve $X$ admits a nonconstant rational map to an elliptic curve $E$ defined over $\Q$ of rank 0. In this case, we use the code of the third-named author and Rouse \cite{JacobJeremycode, maylerouse} to find such a map and compute the rational points of $X$. We first compute a model of \(X\) and Fourier expansions of weight \(2\) cusp forms using Zywina's code \cite{ZywinaModular}.  We attempt to fix a rational base point \(Q \in X(\mathbb{Q})\)  (if none is found, we can attempt to proceed using the cusp at infinity, which need not be rational).

We identify a weight \(2\) cusp form \(f\) whose Hecke eigenvalues for primes \(p \nmid N\) agree with the corresponding Fourier coefficients \(a_p(E)\). We then compute periods of \(f\) and (up to a scalar) match the resulting period lattice with the period lattice \(\Lambda(E')\) of an optimal elliptic curve \(E'\) in the isogeny class of \(E\). We consider the natural identification
\[
\alpha\colon X(\mathbb{C})\longrightarrow\mathbb{H}/(G\cap\SL_2(\mathbb{Z})),
\]
where \(G\) is the subgroup associated to the modular curve \(X\). Using the chosen base point, we obtain an analytic map \(X \to \mathbb{C}/\Lambda(E')\) given by
\[
P\longmapsto \int_{\alpha(Q)}^{\alpha(P)} f \, ds,
\]
and composing with the natural isomorphism \(\mathbb{C}/\Lambda(E') \simeq E'(\mathbb{C})\) gives a map \(X \to E'\) defined over a cyclotomic field. When a rational base point \(Q \in X(\mathbb{Q})\) is available, we translate so that \(Q\) maps to the origin and then recover explicit defining equations for a map \(X \to E'\) defined over \(\mathbb{Q}\).

Although the construction uses numerical and heuristic steps, we then rigorously verify that the resulting polynomials indeed define a nonconstant map \(X \to E'\) as claimed. Once the map is verified, we determine \(X(\mathbb{Q})\) by pulling back each rational point \(T \in E'(\mathbb{Q})\) and determining the rational points on the resulting zero dimensional fibers. When possible, the method described in Step 13 of \cite{maylerouse} (based on \(p\)-adic lifting together with point counting bounds) is used to do this efficiently. Otherwise, we use Magma's built-in commands for determining the rational points of the zero dimensional fibers, which work well when the genus of \(X\) is low but are slow in higher genus.

\subsection{Sieving with Respect to a Projection}\label{subsection:sieving} This method is based on Section 8.3 of \cite{MR4468989}. For some modular curves $X$, we can find an automorphism $i \colon X \to X$ such that the Jacobian $J_X$ of $X$ and the Jacobian of the quotient curve have the same rank.  When this occurs, there exists an abelian variety $V$ of rank $0$ such that $J_X$ is the product of $V$ and the Jacobian of the quotient curve. For any point $P \in X(\Q)$, the point $P-i(P)$ lies in $J_X(\Q)_{\tors}$. Hence, it suffices to compute the preimages of rational points in $J_X(\Q)_{\tors}$ under the map $a \colon X(\Q) \to J_X(\Q)_{\tors}$ given by $P \mapsto P-i(P)$, which is injective away from the fixed points of $i$. 

In practice, it is difficult to determine $J_X(\Q)_{\tors}$. However, by the appendix of Katz's article \cite{MR604840}, if $p$ is an odd prime of good reduction, then $J_X(\Q)_{\tors}$ injects into $J_X(\F_p)$. We can use this to obtain an upper bound on the order of $J_X(\Q)_{\tors}$. After this, we either find a point $P \in X(\Q)$ such that $P-i(P)$ has a given order or we find a prime $p$ and show that, modulo $p$, there is no divisor of the form $P-Q$ of that order for any $P,Q \in X(\F_p)$.

\subsection{The Method of Chabauty and Coleman}

This is a method for determining the rational points on a curve $X$ when the rank of its Jacobian is less than its genus. When this assumption holds, one can define a finite subset of the Jacobian of $X$ over $\Q_p$ that contains the set $X(\Q)$. This method is originally due to Chabauty \cite{chabauty1941points} and was made effective by Coleman \cite{coleman1985effective}. We will apply this to several genus 2 curves. There are two built-in commands in \texttt{Magma} that implement this method.

The first function is \texttt{Chabauty0}. Given the Jacobian of a hyperelliptic curve over $\Q$ that is of rank 0, the function finds all of the rational points on the curve. The second function is \texttt{Chabauty}. Given a point $P$ on the Jacobian of a hyperelliptic curve over $\Q$ of genus $2$, this method employs Chabauty's method in conjunction with the Mordell--Weil sieve to return a subset of the rational points of the curve. If the Jacobian has rank 1 over $\Q$, the full list of rational points is returned. In general, the function returns the set of points whose images in the Jacobian lie in the saturation of the group generated by $P$. For more information about the Mordell--Weil sieve, see the subsequent subsection.

\subsection{Quadratic Chabauty and the Mordell--Weil Sieve}\label{s:quadraticchabauty}
We apply quadratic Chabauty along with the Mordell--Weil sieve to determine the rational points of certain modular curves where the rank of the Jacobian is equal to the genus of the curve. Analogously to the creation of a subset of $J(\Q_p)$ in the case of classical Chabauty--Coleman, quadratic Chabauty uses $p$-adic heights to construct a finite subset of $p$-adic points that contains the set of rational points. For more details, see \cite{MR2512779} and \cite{MR3843370}. Recent work  \cite{bianchi2022rational, finnerty2025quadraticchabautyexperimentsgenus} has specifically addressed the case of bielliptic modular curves of genus $2$ and rank 2, which arises on a few occasions in this paper. 

The Mordell--Weil sieve was originally used in Scharaschkin's Ph.D.\ thesis \cite{scharaschkinthesis} to prove that a curve has no rational points, but it can be used more generally to determine the set of rational points. In particular, it can eliminate $p$-adic points from the output of classical or quadratic Chabauty by proving that they do not correspond to rational points.

Each bielliptic curve of genus $2$ over $\Q$ has a model of the form 
\begin{equation*}
y^2=a_6x^6+a_4x^4+a_2x^2+a_0,\hspace{1cm}a_i\in\Z.
\end{equation*}
Using these models, we perform quadratic Chabauty for the three smallest primes $p_1,p_2,p_3$ of good ordinary reduction. In each case, we obtain a set of $p$-adic points that we wish to prove do not lift to rational points. We denote these sets by $A_1, A_2,$ and $A_3$, respectively. 

Let $\iota$ be the Abel--Jacobi embedding of $X(\Q)$ into its Jacobian $J(\Q)$ with respect to a known base point $b$:
\begin{equation*}
    \iota: X(\Q)\hookrightarrow J(\Q),\hspace{1cm} P\mapsto[P-b].
\end{equation*}
Suppose, toward a contradiction, that a $p$-adic point $P_i\in A_i$ corresponds to a rational point $P\in X(\Q)$. Since the Jacobian of $X$ has rank 2, we can find a torsion point $T \in J(\Q)$ and integers $a_1,a_2$ such that
\begin{equation*}
    \iota(P)=a_1B_1+a_2B_2+T,
\end{equation*}
where $B_1$ and $B_2$ represent the generators of the free part of the Jacobian. 

We know points in $A_i$ only modulo $p_i^{m_i}$ for some integer $m_i$ up to a chosen precision, and we can compute $a_1$ and $a_2$ modulo $p_i^{n_i}$ for some $n_i\leq m_i$. This means that for every point in $A_i$, there are at most $|J(\Q)_{\tors}|$ possibilities for the image of $\iota(P)$ in $J(\Q)/p_i^{n_i}J(\Q)$. Our goal in the Mordell--Weil sieve is to show that these cosets in $J(\Q)/p_i^{n_i} J(\Q)$ cannot arise from any rational point by considering reduction maps modulo multiple primes. We follow the approach of Bianchi and Padurariu \cite{bianchi2022rational}.

Let $I$ be a subset of $\{1,2,3\}$ and $M=\prod_{j\in I} p_j^{n_j}$. By the Chinese remainder theorem, we can enlarge $M$ to $M=M'\prod_{j\in I}p_j^{n_j}$ for any $M'$ coprime to each $p_j$. Note that $M'=1$ always satisfies this condition, but one may need to try different values in order to successfully sieve. Let $S$ be a finite set of primes of good reduction for $X$. The following diagram commutes, where $\alpha_S$ and $\beta_S$ are the natural maps:

\[\begin{tikzcd}
	{X(\Q)} & {J(\Q)/MJ(\Q)} \\
	{\prod_{\ell\in S}X(\F_\ell)} & {\prod_{\ell\in S}J(\F_\ell)/MJ(\F_\ell).}
	\arrow["{\pi\circ\iota}", from=1-1, to=1-2]
	\arrow[from=1-1, to=2-1]
	\arrow["{\alpha_S}", shift left=3, from=1-2, to=2-2]
	\arrow["{\beta_S}", from=2-1, to=2-2]
\end{tikzcd}\]

It suffices to find some set $S$ and some subset $C_M$ of the quotient $J(\Q)/MJ(\Q)$ such that
\[
    \alpha_S(C_M)\cap\beta_S(\prod_{\ell\in S}X(\F_\ell))=\emptyset.
\]

This may require trying both different eligible sets $S$ and auxiliary integers $M'$. We carry out this computation explicitly, with documentation available at \cite{finnertysievecode}. In each case to which we applied this method, the sieve eliminated all points that did not easily lift to rational points. Table \ref{ratpts} shows the resulting affine rational points on each of the curves used in this analysis.

\begin{table}[H]\label{ratpts}
\begin{tabular}{|l|l|l|}
\hline
\textbf{Section} &
\textbf{Affine Model of Curve} & \textbf{Rational Points} \\ \hline
\ref{s:rationalisogeniesdegree5} & 
$y^2 + y = -2x^6 + 9x^2 - 7$     &  $(\pm1 , 0),(\pm1,-1)$\\

\ref{s:rationalisogeniesdegree5}& $y^2 = -12x^5 - 18x^4 + 63x^3 - 18x^2 - 12x$ &  $(0,0)$ \\

\ref{subsec:2isoglattice} & $y^2 = 2x^6 + 2$ &  $(\pm1,\pm2)$ \\
\ref{subsec:2isoglattice}  & $y^2 = x^6 + 3x^4 +3x^2+9$ &  $(\pm1,\pm4),(0,\pm3)$ \\
\ref{subsec:3isoglattice}  & $y^2 = 9x^6 - 99x^4 + 27x^2 - 1$ &  $(\pm\frac{1}{3},\pm\frac{8}{9})$ \\ \hline

\end{tabular}
\caption{Rational points on certain bielliptic curves}
\label{tab:biellipticchabauty}
\end{table}

\section{Possible \texorpdfstring{$\ell$}{ell}-adic Images}\label{s:possibleimages}
\subsection{Strategy and Preliminary Results}\label{ss:strat}
Let $E/\Q$ be a non-CM elliptic curve. Suppose that $E$ admits a rational isogeny of degree $p$, where
\[
    p \in \{2,3,5,7,13\}.
\]
When $p \in \{2,3\}$, we impose the additional assumption that $E$ admits no rational isogeny of any prime degree larger than $p$, in order to reduce the number of cases to consider. The goal of this section is to determine, for each $p \in \{2,3,5,7,13\}$, all possible $\ell$-adic images $G_E(\ell^\infty)$ containing $-I$ that can occur (up to conjugacy). 

Lemos's theorem (Theorem \ref{Thm_Lemos}) gives that the $\ell$-adic representation of $E$ is surjective for all primes $\ell > 37$, so we restrict our attention to primes $\ell \leq 37$. By Theorem 1.1.6 and Table 1 of \cite{MR4468989}, for each prime $p\in\{2,3,5,7,13\}$, the possible $\ell$-adic images of $E$ (with $\ell\le 37$) arise as non-CM rational points on the fiber products $X_0(p)$ with the modular curves in the following table:

\begin{table}[H]
    \centering
    \begin{tabular}{|l|l|}
       \hline \textbf{$\ell$}  & \textbf{Labels} \\ \hline
       5  & $X_{S_4}(5)$ (\mc{5.5.0.a.1}), $X_{0}(5)$ (\mc{5.6.0.a.1}), $X_{\mathrm{ns}}^+(5)$ (\mc{5.10.0.a.1}) \\
       7 & $X_{0}(7)$ (\mc{7.8.0.a.1}), $X_{\mathrm{ns}}^+(7)$ (\mc{7.21.0.a.1}), $X_{\mathrm{sp}}^+(7)$ (\mc{7.28.0.a.1})\\
       11 & $X_{0}(11)$ (\mc{11.12.1.a.1}), $X_{\mathrm{ns}}^+(11)$ (\mc{11.55.1.b.1})\\
       13 & $X_{0}(13)$ (\mc{13.14.0.a.1}), $X_{S_4}(13)$ (\mc{13.91.3.a.1})\\
       17 & $X_{0}(17)$ (\mc{17.18.1.a.1}) \\
       19 & $X_{\mathrm{ns}}^+(19)$ (\mc{19.171.8.a.1})\\
       23 & $X_{\mathrm{ns}}^+(23)$ (\mc{23.253.13.a.1})\\
       29 & $X_{\mathrm{ns}}^+(29)$ (\mc{29.406.24.a.1})\\
       31 & $X_{\mathrm{ns}}^+(31)$ (\mc{31.465.28.a.1})\\
       37 & $X_{0}(37)$ (\mc{37.38.2.a.1}), $X_{\mathrm{ns}}^{+}(37)$ (\mc{37.666.43.a.1})\\ \hline
    \end{tabular}\caption{Curves to consider for $\ell$-adic images} \label{CurvesToConsider}
\end{table}

\begin{remark}
    Several of the modular curves in the table above have no known non-CM rational points. In fact, it is conjectured that for $\ell \geq 19$, the curve $X_{\mathrm{ns}}^+(\ell)$ has no non-CM rational points (see \cite[Conjecture 1.12]{zywina2015possibleimagesmodell}, for instance). However, to obtain an unconditional result, we must still consider fiber products with these modular curves.   
\end{remark}

We now employ some of the earlier results to simplify the analysis. For any positive integer $n$, if $E/\Q$ is a non-CM elliptic curve that admits a rational $n$-isogeny, then every quadratic twist $E'/\Q$ of $E$ also admits a rational $n$-isogeny by Lemma \ref{lem:twistisog}. By Lemma \ref{lem:jinv}, we have that \[ [\GL_2(\widehat{\Z}) : G_E] = [\GL_2(\widehat{\Z}) : G_{E'}].\] This means that for any positive integer $m$, we can always consider a twist $E'$ of $E$ with  $-I \in G_{E'}(m)$ by Lemma \ref{lem:-Itwist}. This greatly reduces the number of possible $\ell$-adic images to consider, since we need only consider those containing $-I$.
We also use the following result of Lemos \cite{MR3885140} throughout this section when considering  fiber products involving $X_0(p)$.

\begin{proposition}\label{P:j-integer}
    Let $p$ be a prime number for which $X_0(p)$ has genus $0$, i.e., $p \in \{2,3,5,7,13\}$.  Let $E/\Q$ be an elliptic curve that admits a rational $p$-isogeny. Suppose there exists a prime $\ell > 5$ not equal to $p$ such that the image of $\rho_{E,\ell} \colon \Gal(\overline{\Q}/\Q) \to \GL_2(\Z/\ell\Z)$ is contained in the normalizer of a nonsplit Cartan subgroup of $\GL_2(\Z/\ell\Z)$. Then $j(E)\in\Z$.
\end{proposition}

\begin{proof}
    The argument in  \cite[Section 3]{MR3885140} shows that the $j$-invariant of $E$ belongs to $\Z[\frac{1}{\ell}]$ as long as $\ell \neq p$. This pair $(E,\ell)$ also satisfies the conditions of  \cite[Proposition 2.2]{MR3885140}, implying that $E$ has potentially good reduction at $\ell$. Hence $j(E)\in\Z[\frac{1}{\ell}] \cap \Z_\ell = \Z$, which completes the proof.
\end{proof}

We now prove two lemmas that will help us rule out non-CM rational points on fiber products with the modular curves $X_{\mathrm{ns}}^+(\ell)$ and $X_{S_4}(13)$. All of the computational claims made in this section can be verified in the file \texttt{Section 4.1.m} in the repository \cite{projectcode}.

\begin{lemma}\label{lemma:fiberproductwithXns+}
    For $p \in \{2,3,5,7,13\}$ and a prime $\ell > 5$ not equal to $p$, $X_0(p) \times X_{\mathrm{ns}}^{+}(\ell)$ has no non-CM rational points.
\end{lemma}

\begin{proof}
    Let $p \in \{2,3,5,7,13\}$ and $\ell > 5$ be a prime distinct from $p$. From Proposition \ref{P:j-integer}, we know that if $X_0(p) \times X_{\mathrm{ns}}^{+}(\ell)$ has any non-CM rational points, then the corresponding $j$-invariant must be an integer. The set of integral $j$-invariants that occur for non-CM rational points of $X_0(p)$ is finite for each $p$ and listed in \cite[p.\ 7]{MR3885140}. We reproduce these $j$-invariants in the table below.
    
    {\small
    \begin{table}[H]
    \centering
    \begin{tabular}{|l|l|l|}
       \hline \textbf{$p$}  & \textbf{Integral $j$-invariants} & \textbf{Representative elliptic curves}\\ \hline
       2  &  $-2^2\cdot7^3 ,-2^4\cdot3^3 ,-2^6,2^7,$ & \ec{7688.c2}, \ec{1800.b2}, \ec{1568.a2}, \ec{128.a2}, \\
       & $2^4\cdot5^3 ,2^{11},2^2\cdot3^6 ,2^7\cdot3^3,$ & \ec{2312.b2}, \ec{200.b2}, \ec{8712.d2}, \ec{1152.d1},\\ 
       & $17^3 ,2^6\cdot5^3,2^5\cdot7^3,2^5\cdot3^6,$ & \ec{4225.e2}, \ec{256.a1}, \ec{128.a1}, \ec{1152.d2}, \\ 
       & $2^4\cdot17^3,2^3\cdot31^3,$ & \ec{200.b1}, \ec{1568.a1}, \\
       & $2^2\cdot3^6\cdot7^3, 2^2\cdot5^3\cdot13^3, 2\cdot127^3, 2\cdot3^3\cdot43^3,$ & \ec{1800.b1}, \ec{2312.b1}, \ec{7688.c1}, \ec{8712.d1}, \\ 
       &  $257^3$ & \ec{4225.e1} \\\hline
       
       3 & $-2^4\cdot11^6\cdot13 ,-2^4 \cdot3^2 \cdot13^3,-2^4\cdot13 ,2^4\cdot3^3,$ & \ec{676.a1}, \ec{324.b1}, \ec{676.a2}, \ec{324.b2}, \\
       & $2^8\cdot7 ,2^4\cdot3^3\cdot5,2^8\cdot3^3 ,2^8\cdot3^2\cdot7^3,$ & \ec{196.a2}, \ec{2700.d2}, \ec{324.a2}, \ec{324.a1}, \\
       & $2^4\cdot3\cdot5\cdot41^3,2^8\cdot7\cdot61^3$ & \ec{2700.d1}, \ec{196.a1} \\ \hline
       5 & $-2^6\cdot719^3,-2^6\cdot5\cdot19^3,2^6,2^6\cdot5^2,2^{12},$ & \ec{43264.f1}, \ec{6400.d1}, \ec{43264.f2}, \ec{6400.d2},\\
       & $2^{12}\cdot5^2,2^{12}\cdot5\cdot11^3,2^{12}\cdot211^3$ & \ec{1369.a2}, \ec{4225.a1}, \ec{4225.a2}, \ec{1369.a1} \\ \hline
       7 & $-3^3\cdot37\cdot719^3,3^3\cdot37,3^2\cdot7^4,3^2\cdot7\cdot2647^3$ & \ec{1369.b1}, \ec{1369.b2}, \ec{3969.a2}, \ec{3969.a1} \\ \hline 
       13 & $-2^6\cdot3^2\cdot4079^3, 2^6\cdot3^2, 2^{12}\cdot3^3\cdot19, 2^{12}\cdot3^3\cdot19\cdot991^3$ & \ec{20736.a1}, \ec{20736.a2}, \ec{9025.a2},\ec{9025.a1} \\ \hline
       
    \end{tabular}\caption{Summary of $j$-invariants to consider for proof of Lemma \ref{lemma:fiberproductwithXns+}}
\end{table}}
    
    For each of these $j$-invariants, we consider a representative elliptic curve over $\Q$ (also listed in the table) and observe that its $\ell$-adic representation is surjective. Therefore, the modular curve $X_0(p) \times X_{\mathrm{ns}}^+(\ell)$ cannot have any non-CM rational points.
\end{proof}

\begin{lemma}\label{lemma:fiberproductwithXS4(13)}
    If $E/\Q$ corresponds to a non-CM rational point of $X_{S_4}(13)$, then for any prime $p \neq 13$, the $p$-adic representation of $E$ is surjective.
\end{lemma}

\begin{proof}
    From Section 5.1 of \cite{MR4589060}, we know that the modular curve $X_{S_4}(13)$ (\mc{13.91.3.a.1}) has three non-CM $j$-invariants:
\begin{equation*}
    \frac{2^4\cdot5\cdot13^4\cdot17^3}{3^{13}}, \frac{-2^{12}\cdot5^3\cdot11\cdot13^4}{3^{13}}, \frac{2^{18}\cdot3^3\cdot13^4\cdot127^3\cdot139^3\cdot157^3\cdot283^3\cdot929}{5^{13}\cdot61^{13}}.
\end{equation*}
For each of these $j$-invariants, we consider a representative elliptic curve $E/\Q$. The LMFDB labels of the first two representatives are \ec{50700.z1} and \ec{61347.bb1}, respectively. Let $p \neq 13$ be a prime. For both of these curves, we see that their $p$-adic representation is surjective. Hence, if any elliptic curve $E/\Q$ has either of these two $j$-invariants, then its $p$-adic representation must also be surjective (this follows by \cite[Lemma 5]{GMR_smallestSurjectivePrime}). The result follows similarly for the third $j$-invariant; however, there is no elliptic curve in the LMFDB with this $j$-invariant, so we instead construct one using the \texttt{Magma} command \texttt{EllipticCurveFromjInvariant} and run Zywina's \texttt{FindOpenImage} \cite{ZywinaOpen} to compute its Galois image. We find that the only prime divisors of its adelic level are among $2$, $5$, $13$, $61$ and $109$. If $p$ is any of these primes (except $13$), we check that the $p$-adic representation is surjective, and if $p$ is outside this list, then necessarily the $p$-adic representation is surjective.
\end{proof}

\subsection{Computation of Possible Images}

We are now ready to begin the case-by-case classification of  $\ell$-adic images for curves admitting a rational isogeny of degree $p\in\{2,3,5,7,13\}$. We begin with the case $p = 2$.

\begin{proposition}\label{P:2isog-elladic}
 Let $E/\Q$ be a non-CM elliptic curve that admits a rational isogeny of degree $2$ and no rational isogeny of prime degree greater than $2$. 

 \begin{enumerate}
     \item If the $2$-adic Galois image of $E$ contains $-I$, then it is conjugate to one of the following $188$ subgroups of $\GL_2(\Z_2)$:

\begin{multicols}{4}
    \begin{itemize}
\item \mc{2.3.0.a.1}
\item \mc{2.6.0.a.1}
\item \mc{4.6.0.a.1}
\item \mc{4.6.0.b.1}
\item \mc{4.6.0.c.1}
\item \mc{4.6.0.d.1}
\item \mc{4.6.0.e.1}
\item \mc{4.12.0.a.1}
\item \mc{4.12.0.b.1}
\item \mc{4.12.0.d.1}
\item \mc{4.12.0.e.1}
\item \mc{4.12.0.f.1}
\item \mc{4.24.0.b.1}
\item \mc{4.24.0.c.1}
\item \mc{8.6.0.a.1}
\item \mc{8.6.0.b.1}
\item \mc{8.6.0.c.1}
\item \mc{8.6.0.d.1}
\item \mc{8.6.0.e.1}
\item \mc{8.6.0.f.1}
\item \mc{8.12.0.a.1}
\item \mc{8.12.0.b.1}
\item \mc{8.12.0.d.1}
\item \mc{8.12.0.f.1}
\item \mc{8.12.0.g.1}
\item \mc{8.12.0.h.1}
\item \mc{8.12.0.i.1}
\item \mc{8.12.0.k.1}
\item \mc{8.12.0.l.1}
\item \mc{8.12.0.m.1}
\item \mc{8.12.0.n.1}
\item \mc{8.12.0.o.1}
\item \mc{8.12.0.p.1}
\item \mc{8.12.0.q.1}
\item \mc{8.12.0.r.1}
\item \mc{8.12.0.s.1}
\item \mc{8.12.0.t.1}
\item \mc{8.12.0.u.1}
\item \mc{8.12.0.v.1}
\item \mc{8.12.0.w.1}
\item \mc{8.12.0.x.1}
\item \mc{8.12.0.y.1}
\item \mc{8.12.0.z.1}
\item \mc{8.24.0.a.1}
\item \mc{8.24.0.ba.1}
\item \mc{8.24.0.ba.2}
\item \mc{8.24.0.bb.1}
\item \mc{8.24.0.bb.2}
\item \mc{8.24.0.bc.1}
\item \mc{8.24.0.be.1}
\item \mc{8.24.0.bf.1}
\item \mc{8.24.0.bh.1}
\item \mc{8.24.0.bi.1}
\item \mc{8.24.0.bj.1}
\item \mc{8.24.0.bl.1}
\item \mc{8.24.0.bl.2}
\item \mc{8.24.0.bn.1}
\item \mc{8.24.0.bo.1}
\item \mc{8.24.0.bp.1}
\item \mc{8.24.0.bq.1}
\item \mc{8.24.0.bs.1}
\item \mc{8.24.0.bt.1}
\item \mc{8.24.0.c.1}
\item \mc{8.24.0.d.1}
\item \mc{8.24.0.d.2}
\item \mc{8.24.0.e.1}
\item \mc{8.24.0.e.2}
\item \mc{8.24.0.f.1}
\item \mc{8.24.0.g.1}
\item \mc{8.24.0.h.1}
\item \mc{8.24.0.i.1}
\item \mc{8.24.0.k.1}
\item \mc{8.24.0.m.1}
\item \mc{8.24.0.n.1}
\item \mc{8.24.0.o.1}
\item \mc{8.24.0.q.1}
\item \mc{8.24.0.r.1}
\item \mc{8.24.0.s.1}
\item \mc{8.24.0.t.1}
\item \mc{8.24.0.v.1}
\item \mc{8.24.0.x.1}
\item \mc{8.24.0.y.1}
\item \mc{8.24.0.z.1}
\item \mc{8.48.0.b.1}
\item \mc{8.48.0.b.2}
\item \mc{8.48.0.c.1}
\item \mc{8.48.0.e.1}
\item \mc{8.48.0.e.2}
\item \mc{8.48.0.f.1}
\item \mc{8.48.0.h.1}
\item \mc{8.48.0.h.2}
\item \mc{8.48.0.i.1}
\item \mc{8.48.0.k.1}
\item \mc{8.48.0.k.2}
\item \mc{8.48.0.l.1}
\item \mc{8.48.0.l.2}
\item \mc{8.48.0.m.1}
\item \mc{8.48.0.m.2}
\item \mc{8.48.0.n.1}
\item \mc{8.48.0.n.2}
\item \mc{8.48.0.p.1}
\item \mc{8.48.0.q.1}
\item \mc{8.48.0.q.2}
\item \mc{16.24.0.a.1}
\item \mc{16.24.0.c.1}
\item \mc{16.24.0.e.1}
\item \mc{16.24.0.e.2}
\item \mc{16.24.0.f.1}
\item \mc{16.24.0.f.2}
\item \mc{16.24.0.g.1}
\item \mc{16.24.0.h.1}
\item \mc{16.24.0.i.1}
\item \mc{16.24.0.j.1}
\item \mc{16.24.0.k.1}
\item \mc{16.24.0.k.2}
\item \mc{16.24.0.l.1}
\item \mc{16.24.0.l.2}
\item \mc{16.24.0.m.1}
\item \mc{16.24.0.m.2}
\item \mc{16.24.0.n.1}
\item \mc{16.24.0.n.2}
\item \mc{16.24.0.o.1}
\item \mc{16.24.0.o.2}
\item \mc{16.24.0.p.1}
\item \mc{16.24.0.p.2}
\item \mc{16.48.0.a.1}
\item \mc{16.48.0.bb.1}
\item \mc{16.48.0.bb.2}
\item \mc{16.48.0.c.1}
\item \mc{16.48.0.c.2}
\item \mc{16.48.0.d.1}
\item \mc{16.48.0.d.2}
\item \mc{16.48.0.e.1}
\item \mc{16.48.0.h.1}
\item \mc{16.48.0.h.2}
\item \mc{16.48.0.i.1}
\item \mc{16.48.0.j.1}
\item \mc{16.48.0.l.1}
\item \mc{16.48.0.l.2}
\item \mc{16.48.0.m.1}
\item \mc{16.48.0.m.2}
\item \mc{16.48.0.n.1}
\item \mc{16.48.0.q.1}
\item \mc{16.48.0.r.1}
\item \mc{16.48.0.t.1}
\item \mc{16.48.0.t.2}
\item \mc{16.48.0.u.1}
\item \mc{16.48.0.u.2}
\item \mc{16.48.0.v.1}
\item \mc{16.48.0.v.2}
\item \mc{16.48.0.x.1}
\item \mc{16.48.0.x.2}
\item \mc{16.48.0.y.1}
\item \mc{16.48.0.y.2}
\item \mc{16.48.0.z.1}
\item \mc{16.48.0.z.2}
\item \mc{16.48.1.bg.1}
\item \mc{16.48.1.bl.1}
\item \mc{16.48.1.bn.1}
\item \mc{16.48.1.bq.1}
\item \mc{16.48.1.bs.1}
\item \mc{16.48.1.bv.1}
\item \mc{16.48.1.ca.1}
\item \mc{16.48.1.cc.1}
\item \mc{16.48.1.cd.1}
\item \mc{16.48.1.cf.1}
\item \mc{16.48.1.cg.1}
\item \mc{16.48.1.ch.1}
\item \mc{16.48.1.cr.1}
\item \mc{16.48.1.cs.1}
\item \mc{16.48.1.ct.1}
\item \mc{16.48.1.cx.1}
\item \mc{16.48.1.cx.2}
\item \mc{16.48.1.dc.1}
\item \mc{16.48.1.de.1}
\item \mc{16.48.1.df.1}
\item \mc{16.96.3.ey.1}
\item \mc{16.96.3.ey.2}
\item \mc{16.96.3.fa.1}
\item \mc{16.96.3.fa.2}
\item \mc{32.48.0.a.1}
\item \mc{32.48.0.c.1}
\item \mc{32.48.0.e.1}
\item \mc{32.48.0.e.2}
\item \mc{32.48.0.f.1}
\item \mc{32.48.0.f.2}
\item \mc{32.96.3.bh.1}
\item \mc{32.96.3.bs.1}
    \end{itemize}
\end{multicols}

     \item If the $3$-adic Galois image of $E$ contains $-I$, then it is conjugate to one of the following subgroups of $\GL_2(\Z_3)$: $\GL_2(\Z_3)$, \mc{3.3.0.a.1}, or  \mc{3.6.0.b.1}.

     \item The $\ell$-adic Galois representation of $E$ is surjective for all primes $\ell \not\in \{2,3\}$.
 \end{enumerate}

\end{proposition}

\begin{proof}

Part (1) follows from Tables 11-14 of \cite{MR4868206}. Part (2) follows from Theorem 1.1 of \cite{rakvi2023classification3adicgalois}. Part (3) follows for primes $\ell > 37$ by Theorem \ref{Thm_Lemos}, so it remains to treat the primes $5 \le \ell \le 37$. We proceed along the lines of Table \ref{CurvesToConsider}, using the strategy set forth in Section \ref{ss:strat}.

\emph{Proving $5$-adic surjectivity.} By Table \ref{CurvesToConsider}, in order to prove that the $5$-adic Galois representation of  $E$ is surjective, it suffices to study the modular curves
\[
X_0(2) \times X_{S_4}(5), \quad
X_0(2) \times X_0(5), \quad \text{and} \quad
X_0(2) \times X_{\mathrm{ns}}^+(5).
\]
By our assumption that $E$ does not admit a rational isogeny of prime degree greater than $2$, we may disregard the fiber product $X_0(2) \times X_0(5)$. The curves $X_0(2) \times X_{S_4}(5)$ and $X_0(2) \times X_{\mathrm{ns}}^+(5)$ have LMFDB labels \mc{10.15.1.a.1} and \mc{10.30.1.a.1}, respectively. They both have genus 1 and rank 0. We check that each has no non-CM rational points.  It follows that the $5$-adic Galois representation of $E$ is surjective.

\emph{Proving $7$-adic surjectivity.} By Table \ref{CurvesToConsider}, in order to prove that the $7$-adic Galois representation of  $E$ is surjective, it suffices to study the modular curves
\[
X_0(2) \times X_0(7), \quad
X_0(2) \times X_{\mathrm{ns}}^+(7), \quad \text{and} \quad
X_0(2) \times X_{\mathrm{sp}}^+(7).
\]
Again, by our isogeny assumption, we do not need to consider the fiber product $X_0(2) \times X_0(7)$. The curve $X_0(2) \times X_{\mathrm{ns}}^+(7)$ (\mc{14.63.2.a.1}) has no non-CM rational points by Lemma \ref{lemma:fiberproductwithXns+}. The curve $X_0(2) \times X_{\mathrm{sp}}^{+}(7)$ (\mc{14.84.3.a.1}) has genus $3$ and rank $0$. It admits a map to a rank $0$ elliptic curve. Running the code \texttt{ModCrvToEC} \cite{JacobJeremycode} (discussed in Section \ref{S:M2EC}), we determine that this curve 
has exactly four rational points, all of which are cusps or CM points. Hence the $7$-adic Galois representation of $E$ is surjective.

\emph{Proving $\ell$-adic surjectivity for $11 \leq \ell \leq 37$.} By Table \ref{CurvesToConsider}, all that remains is to study the curves
\begin{align*}
X_0(2) &\times X_0(\ell) \;\; \text{for} \;\; \ell \in \{11, 13, 17, 37\}, \\
X_0(2) &\times X_{\mathrm{ns}}^+(\ell) \;\; \text{for} \;\; \ell \in \{11, 19, 23, 29, 31, 37\}, \text{ and} \\
X_0(2) &\times X_{S_4}(13).
\end{align*}
Again, by our isogeny assumption, we do not need to consider the curves of the form $X_0(2) \times X_0(\ell)$. The curves $X_0(2) \times X_{\mathrm{ns}}^+(\ell)$ have no non-CM rational points by Lemma \ref{lemma:fiberproductwithXns+}, and $X_0(2) \times X_{S_4}(13)$ has no non-CM rational points by Lemma \ref{lemma:fiberproductwithXS4(13)}. Thus, the $\ell$-adic Galois representation of $E$ is surjective for all primes $11 \leq \ell \leq 37$.

All computations can be verified in the file \texttt{2 isogeny l adic images.m} in the repository \cite{projectcode}.
\end{proof}

The result for a rational isogeny of degree 3 follows a similar structure. We have the following:

\begin{proposition}  \label{P:3isog-elladic}
 Let $E/\Q$ be a non-CM elliptic curve that admits a rational isogeny of degree $3$ and no rational isogeny of prime degree greater than $3$. 
    \begin{enumerate}
        \item If the $2$-adic Galois image of $E$ contains $-I$, then it is conjugate to one of the following subgroups of $\GL_2(\Z_2)$: $\GL_2(\Z_2)$, \mc{2.2.0.a.1},
        \mc{2.3.0.a.1},
        \mc{2.6.0.a.1},
        \mc{4.2.0.a.1},
        \mc{4.6.0.a.1}, 
        \mc{4.6.0.b.1}, \mc{4.6.0.c.1},
        \mc{4.8.0.b.1},
        \mc{8.2.0.a.1}, \mc{8.2.0.b.1}, 
        \mc{8.6.0.a.1},
        \mc{8.6.0.b.1}, \\
        \mc{8.6.0.c.1}, or \mc{8.6.0.d.1}.
          
        \item If the $3$-adic Galois image of $E$ contains $-I$, then it is conjugate to one of the following subgroups of $\GL_2(\Z_3)$: \mc{3.4.0.a.1}, \mc{3.12.0.a.1}, \mc{9.12.0.a.1}, \mc{9.12.0.b.1}, \mc{9.36.0.a.1}, \mc{9.36.0.b.1}, \mc{9.36.0.c.1}, \mc{9.36.0.d.1}, \mc{9.36.0.d.2}, \mc{9.36.0.e.1}, \mc{9.36.0.f.1},\\ \mc{9.36.0.f.2}, \mc{9.36.0.g.1},  or \mc{27.36.0.a.1}.     
       
        \item If the $5$-adic Galois image of $E$ contains $-I$, then it is conjugate to either $\GL_2(\Z_5)$ or \mc{5.5.0.a.1}.

        \item The $\ell$-adic Galois representation of $E$ is surjective for all primes $\ell \not\in \{2,3,5\}$. 
    \end{enumerate}
\end{proposition}

\begin{proof}

 Part (1) follows from Tables 15-19 of \cite{MR4868206}. Part (2) follows from Theorem 1.1 of \cite{rakvi2023classification3adicgalois}. For primes $\ell > 37$, part (4) follows from Theorem 1.1 of \cite{MR3885140}.

 \emph{Considering $5$-adic images.} By our isogeny assumption, we do not need to consider the fiber product $X_0(3) \times X_0(5)$. The curve $X_0(3) \times X_{\mathrm{ns}}^{+}(5)$ (\mc{15.40.2.a.1}) is  hyperelliptic of genus $2$ and rank $0$. Using the command \texttt{Chabauty0} in \texttt{Magma}, we determine that it has exactly five rational points, all of which correspond to CM elliptic curves. Finally, the fiber product $X_0(3) \times X_{S_4}(5)$ (\mc{15.20.1.a.1}) is an elliptic curve of rank 0. It has exactly five rational points: two are cusps, one is CM, and two are non-CM. Their $j$-invariants are $2^{4} \cdot 3^{3}$ and $ - 2^{4} \cdot 3^{2} \cdot 13^{3}$. For each of these, we consider representative elliptic curves: \ec{324.b2} and \ec{324.b1}, respectively. For both of these elliptic curves, their $5$-adic image is $X_{S_4}(5)$ (\mc{5.5.0.a.1}). By Lemma \ref{lem:-Itwist}, any of its quadratic twists will also have the same $5$-adic image, because there are no index $2$ subgroups of $X_{S_4}(5)$ that can occur as $5$-adic images.

 \emph{Proving $7$-adic surjectivity.} By our isogeny assumption, we do not need to consider $X_0(3) \times X_0(7)$. The fiber products $X_0(3) \times X_{\mathrm{ns}}^{+}(7)$ (\mc{21.84.5.a.1}) and $X_0(3) \times X_{\mathrm{sp}}^{+}(7)$ (\mc{21.112.6.a.1}) both map to rank $0$ elliptic curves. We again use the code \texttt{ModCrvToEC} for these. There is exactly one rational point on \mc{21.84.5.a.1}, and it is a CM point. There are exactly seven rational points on \mc{21.112.6.a.1}, all of which are cusps or CM points.

\emph{Proving $\ell$-adic surjectivity for $11 \le \ell \le 37$.} By our isogeny assumption, we do not need to consider the curves $X_0(3) \times X_0(\ell)$. For primes $\ell$ in this interval, $X_0(3) \times X_{\mathrm{ns}}^+(\ell)$ has no non-CM rational points by Lemma \ref{lemma:fiberproductwithXns+}. From Lemma \ref{lemma:fiberproductwithXS4(13)}, we know that $X_0(3) \times X_{S_4}(13)$ has no non-CM rational points.

All computations can be verified in the file \texttt{3 isogeny l adic images.m} in the repository \cite{projectcode}. 
\end{proof}

In the next proposition, we address the case that $E/\Q$ admits a rational isogeny of degree 5.

\begin{proposition}  \label{P:5isog-elladic}
    Let $E/\Q$ be a non-CM elliptic curve  that admits a rational isogeny of degree $5$.
    \begin{enumerate}
        \item If the $2$-adic Galois image of $E$ contains $-I$, then it is conjugate to one of the following subgroups of $\GL_2(\Z_2)$:
        $\GL_2(\Z_2)$, \mc{2.3.0.a.1}, \mc{4.4.0.a.1}, \mc{8.2.0.a.1}, \mc{8.6.0.a.1}, \mc{8.6.0.d.1}, or \mc{8.6.0.f.1}.
        
        \item If the $3$-adic Galois image of $E$ contains $-I$, then it is conjugate to one of the following subgroups of $\GL_2(\Z_3)$: $\GL_2(\Z_3)$, \mc{3.3.0.a.1}, \mc{3.4.0.a.1}, \mc{3.6.0.b.1}, or \mc{9.9.0.a.1}. 
       
        \item If the $5$-adic Galois image of $E$ contains $-I$, then it is conjugate to one of the following subgroups of $\GL_2(\Z_5)$: \mc{5.6.0.a.1}, \mc{5.12.0.a.1}, \mc{5.12.0.a.2}, \mc{5.30.0.a.1}, \mc{5.60.0.a.1}, \mc{25.30.0.a.1}, \mc{25.60.0.a.1}, or \mc{25.60.0.a.2}.
        
        \item The $\ell$-adic Galois representation of $E$ is surjective for all primes $\ell \not\in \{2,3,5\}$. 
    \end{enumerate}
\end{proposition}
\begin{proof}
    Part (1) follows from Tables 17-19 of \cite{MR4868206}. Part (2) follows from Theorem 1.1 of \cite{rakvi2023classification3adicgalois}. Part (3) follows from Theorem 1.1.6 of \cite{MR4468989}. Part (4) follows for primes $\ell > 37$ by Theorem \ref{Thm_Lemos}.
    
    \emph{Proving $7$-adic surjectivity.} The fiber product $X_0(5) \times X_0(7)$ has label \mc{35.48.3.a.1}; it has no non-CM rational points by Theorem \ref{thm:isogeniescomposite}. Next, we discuss the curve $ X_0(5) \times X_{\mathrm{ns}}^{+}(7)$ (\mc{35.126.6.a.1}). Let $C$ denote the subscheme of $\mathbb{P}^3_{\Q}$ defined by the sequence of homogeneous polynomials giving the canonical model of \mc{35.126.6.a.1}. We verify that the set $C(\Z/49\Z)$ is empty; hence, there are no rational points on $X_0(5) \times X_{\mathrm{ns}}^{+}(7)$. Lastly, seeking a contradiction, suppose that $E/\Q$  corresponds to a rational point on the fiber product $X_0(5) \times X_{\mathrm{sp}}^{+}(7)$ (\mc{35.168.9.a.1}). Then there exists a quadratic field $K$ such that $E/K$ corresponds to a quadratic point on the fiber product $X_0(5) \times X_{\mathrm{sp}}(7)$. In particular, $E/K$ corresponds to a quadratic point on the fiber product $X_0(5) \times X_{0}(7)$. Moreover, since $E$ is defined over $\Q$, its $j$-invariant lies in $\Q$. But by Theorem $1$ of \cite{Vux91}, this is impossible.
        
    \emph{Proving $\ell$-adic surjectivity for $11 \le \ell \le 37$.} For primes $\ell$ in this interval, $X_0(5) \times X_0(\ell)$ and $X_0(5) \times X_{\mathrm{ns}}^+(\ell)$ have no non-CM rational points by Theorem \ref{thm:isogeniescomposite} and Lemma \ref{lemma:fiberproductwithXns+}, respectively. From Lemma \ref{lemma:fiberproductwithXS4(13)}, we know that $X_0(5) \times X_{S_4}(13)$ has no non-CM rational points.

All computations can be verified in the file \texttt{5 isogeny l adic images.m} in the repository \cite{projectcode}.
\end{proof}

We now classify Galois images for non-CM elliptic curves admitting a rational isogeny of degree 7.

\begin{proposition} \label{P:7isog-elladic}
    Let $E/\Q$ be a non-CM elliptic curve  that admits a rational isogeny of degree $7$.
    \begin{enumerate}
        \item If the $2$-adic Galois image of $E$ contains $-I$, then it is conjugate to one of the following subgroups of $\GL_2(\Z_2)$: $\GL_2(\Z_2)$, 
            \mc{2.2.0.a.1}, 
            \mc{4.2.0.a.1},   
            \mc{4.8.0.b.1},  
            \mc{8.2.0.a.1}, or 
            \mc{8.2.0.b.1}.
        
        \item If the $3$-adic Galois image of $E$ contains $-I$, then it is conjugate to either $\GL_2(\Z_3)$ or
            \mc{3.4.0.a.1}.
        
        \item If the $7$-adic Galois image of $E$ contains $-I$, then it is conjugate to one of the following four subgroups of $\GL_2(\Z_7)$:
    \mc{7.8.0.a.1},
            \mc{7.24.0.a.1},
            \mc{7.24.0.a.2}, or 
            \mc{7.24.0.b.1}.
        
        \item The $\ell$-adic Galois representation of $E$ is surjective for all primes $\ell \not\in \{2,3,7\}$. 
    \end{enumerate}
\end{proposition}
\begin{proof}
Part (1) follows from Table 17 and Table 19 of \cite{MR4868206}. Part (2) follows from Theorem 1.1 of \cite{rakvi2023classification3adicgalois}. Part (3) follows from Theorem 1.1.6 of \cite{MR4468989}. For primes $\ell > 37$, part (4) follows from Theorem 1.1 of \cite{MR3885140}. Thus, it only remains to verify part (4) for primes $\ell \le 37$. 

\emph{Proving $5$-adic surjectivity.} First, the fiber product $X_0(7) \times X_{S_4}(5)$ (\mc{35.40.2.a.1}) is a hyperelliptic curve of genus $2$ and rank $0$. Executing \texttt{Chabauty0} in \texttt{Magma} shows that there are only two rational points, both cusps. Next, the fiber product $X_0(5) \times X_0(7)$ (\mc{35.48.3.a.1}) has no non-CM rational points by Theorem \ref{thm:isogeniescomposite}. 

Let us next consider the fiber product $X_0(7) \times X_{\mathrm{ns}}^+(5)$ (\mc{35.80.5.a.1}). This curve, which we denote by $C$, has genus 5 and rank 1. We use the canonical model in $\mathbb{P}^4$ given on LMFDB. It has two known rational points $P_1=(1:-6:3:3:5)$ and $P_2=(1:-2:3:7:-5)$; these both correspond to CM elliptic curves. We will now prove that these are the only rational points. There exists an automorphism $i \colon C(\Q) \to C(\Q)$ given by \[(x:y:z:w:t) \mapsto (x: y: z: w: x-2y+2z-3w-t).\] The quotient of $C$ by this automorphism is an elliptic curve $E$ with rank $1$. Therefore, there exists an abelian variety of rank $0$, which we denote $V$, such that the Jacobian of $C$, denoted by $J_C$, decomposes as $E \times V$. 

For any point $P \in C(\Q)$, the point $P-i(P)$ lies in $J_C(\Q)_{\tors}$. It thus suffices to compute preimages of rational points in $J_C(\Q)_{\tors}$ under the map $a \colon C(\Q) \to J_C(\Q)_{\tors}$ given by $P \mapsto P-i(P)$. Note that $a$ is injective away from the fixed points of $i$. By using local computations, we can deduce that $J_C(\Q)_{\tors}$ is a subgroup of $\Z/2\Z \times \Z/10\Z$. Further, by working modulo $3$, we check that there is no divisor of the form $P -Q$ that has order $2$ or $5$. Working modulo $13$, we see that there is no divisor of the form $P-Q$ that has order $10$. Therefore, all rational points of $C$ are fixed points of this automorphism. We find the fixed points in \texttt{Magma}. The fixed points of $i$ of this modular curve form a scheme of dimension $0$. There are exactly two such rational points: $P_1$ and $P_2$, as claimed.

    \emph{Proving $\ell$-adic surjectivity for $11 \le \ell \le 37$.} For primes $\ell$ in this interval, $X_0(7) \times X_0(\ell)$ and $X_0(7) \times X_{\mathrm{ns}}^+(\ell)$ have no non-CM rational points by Theorem \ref{thm:isogeniescomposite} and Lemma \ref{lemma:fiberproductwithXns+}, respectively. From Lemma \ref{lemma:fiberproductwithXS4(13)}, we know that $X_0(7) \times X_{S_4}(13)$ has no non-CM rational points.

All computations can be verified in the file \texttt{7 isogeny l adic images.m} in the repository \cite{projectcode}.
\end{proof}

Lastly, we classify the $\ell$-adic images of non-CM elliptic curves that admit a rational isogeny of degree 13.

\begin{proposition}  \label{P:13isog-elladic}
    Let $E/\Q$ be a non-CM elliptic curve  that admits a rational isogeny of degree $13$.
    \begin{enumerate}
        \item If the $2$-adic Galois image of $E$ contains $-I$, then it is conjugate to either $\GL_2(\Z_2)$ or \mc{8.2.0.a.1}.
        
        \item If the $13$-adic Galois image of $E$ contains $-I$, then it is conjugate to one of the following subgroups of $\GL_2(\Z_{13})$: \mc{13.14.0.a.1}, \mc{13.28.0.a.1}, \mc{13.28.0.a.2}, \mc{13.42.0.a.1}, \mc{13.42.0.a.2}, or  \mc{13.42.0.b.1}.
        
        \item The $\ell$-adic Galois representation of $E$ is surjective for all primes $\ell \not\in \{2,13\}$. 
    \end{enumerate}
\end{proposition}

\begin{proof} Part (1) follows from Table 19 of \cite{MR4868206}. Part (2) follows from Theorem 1.1.6 of \cite{MR4468989}. In the case that $\ell = 3$, part (3) follows from Theorem 1.1 of \cite{rakvi2023classification3adicgalois}. For primes $\ell > 37$, part (3) follows from Theorem 1.1 of \cite{MR3885140}. Thus, it only remains to verify part (3) for primes $\ell \le 37$. 

\emph{Proving $5$-adic surjectivity.} The fiber product  $X_0(13) \times X_{S_4}(5)$ (\mc{65.70.4.a.1}) has genus $4$ and rank $1$. It admits a map to a rank 0 elliptic curve, which we find using the code \texttt{ModCrvToEC}. In this way, we determine that the curve has two rational points, both of which are cusps. Next, the curve $X_0(13) \times X_0(5)$ (\mc{65.84.5.a.1}) has no non-CM rational points by Theorem \ref{thm:isogeniescomposite}. Lastly, the curve $X_0(13) \times X_{\mathrm{ns}}^+(5)$ (\mc{65.140.9.a.1}) has no rational points modulo $8$ and hence no rational points.

\emph{Proving $7$-adic surjectivity.} First, the fiber product $X_0(13) \times X_0(7)$ (\mc{91.112.7.a.1}) has no non-CM rational points by Theorem \ref{thm:isogeniescomposite}. Next, the fiber product $X_0(13) \times X_{\mathrm{ns}}^+(7)$ has no non-CM rational points by Lemma \ref{lemma:fiberproductwithXns+}. Lastly, suppose that $E/\Q$ corresponds to a rational point on the fiber product $X_0(13) \times X_{\mathrm{sp}}^+(7)$. Then there exists a quadratic field $K$ such that $E/K$  corresponds to a quadratic point on the fiber product $X_{\mathrm{sp}}(7) \times X_0(13)$. Therefore, $E/K$  corresponds to a quadratic point on the fiber product $X_{0}(7) \times X_0(13)$. Moreover, since $E$ is defined over $\Q$ its $j$-invariant lies in $\Q$. From Theorem $1$ of \cite{Vux91}, this is impossible.
        
\emph{Proving $\ell$-adic surjectivity for $\ell = 11$ and $17 \le \ell \le 37$.} For such primes $\ell$, $X_0(13) \times X_0(\ell)$ and $X_0(13) \times X_{\mathrm{ns}}^+(\ell)$ have no non-CM rational points by Theorem \ref{thm:isogeniescomposite} and Lemma \ref{lemma:fiberproductwithXns+}, respectively.

All of the computations can be verified in the file \texttt{13 isogeny l adic images.m} in the repository \cite{projectcode}.
\end{proof}

\section{A Product Formula for the Adelic Index}\label{s:adelicindex}

In this section, we give a product formula for the adelic index of a non-CM elliptic curve that admits a rational isogeny of degree $p \in \{2,3,5,7,13\}$, expressed in terms of the $6$-adic and $p$-adic commutator indices of the Galois image. The main statement is Proposition \ref{P:6-p}.

We begin with a result of Jones \cite{MR3350106} for checking equality of commutator subgroups of open subgroups of $\GL_2(\widehat{\Z})$. Recall that $G'\coloneqq[G, G]$ denotes the commutator subgroup of $G$. The commutator subgroup of an open subgroup of $\GL_2(\widehat{\Z})$ is necessarily open in $\SL_2(\widehat{\Z})$ by \cite[Lemma 7.10]{ZywinaOpen}.

\begin{theorem}[Theorem 2.7 and Remark 2.8, \cite{MR3350106}]\label{P:m0} Let $H$ and $G$ be open subgroups of $\GL_2(\widehat{\Z})$ such that $H \subseteq G$. Let $m$ be any multiple of the level of $G$ as a subgroup of $\GL_2(\widehat{\Z})$. Assume that each prime $\ell$ dividing $m$ satisfies $\ell \not\equiv \pm 1 \pmod{5}$. Set
\[
m_0 \coloneqq \lcm\left(2^3 \cdot 3^3, \prod_{\ell \mid m} \ell^{2\ord_\ell(m) + 1}\right),
\]
where $\ord_{\ell}(m)$ denotes the exact power of $\ell$ dividing $m$. Then, $H' = G'$ if and only if
\begin{enumerate}
	\item for each prime $\ell \nmid m_0$, one has $\SL_2(\F_\ell) \subseteq H(\bmod\,\ell)$, and
	\item one has $H(\bmod\,m_0)' = G(\bmod\,m_0)'$.
\end{enumerate}

\end{theorem}

We now record a lemma that will be used in the proof of Proposition \ref{P:6-p}. For a profinite group $G$, we write $\Quo(G)$ to denote the set of isomorphism classes of quotients of $G$, and we write $\rad(n)$ for the product of the distinct prime divisors of an integer $n$.

\begin{lemma} \label{L:CommonQuo}
    Let $p \geq 5$ be a prime. If $G \subseteq \GL_2(\Z_6)$ and $H \subseteq B_0(p^\infty)$ are open subgroups, then
    \[    \Quo(G') \cap \Quo(H') = \{1\}.    \]
\end{lemma}
\begin{proof} 
    For all $k \geq 1$, we have that
\[
    \rad (|\GL_2(\Z/6^k \Z)'|) = 6 \quad \text{and} \quad 
    \rad (|B_0(p^k)'|) = p.
\]
Since $G(6^k)\subseteq \GL_2(\Z/6^k\Z)$, we also have $G(6^k)'\subseteq \GL_2(\Z/6^k\Z)'$. This gives us that
$|G(6^k)'|$ divides $|\GL_2(\Z/6^k\Z)'|$. Likewise, $H(p^k)\subseteq B_0(p^k)$ implies $|H(p^k)'|$ divides $|B_0(p^k)'|$. Now, suppose that $Q\in \Quo(G')\cap \Quo(H')$. Thus, $Q$ must be finite and its order must divide some power of $\gcd(6,p)$. Since $\gcd(6,p) = 1$, this means that $Q$ must be the trivial group.
\end{proof}

We next recall the fiber product construction for groups. Let $G_1$, $G_2$, and $Q$ be groups, and let  $\phi_1 \colon G_1 \to Q$ and $\phi_2 \colon G_2 \to Q$ be surjective group homomorphisms. The \emph{fiber product of $G_1$ and $G_2$ by $(\phi_1, \phi_2)$} is the group  
\[
G_1 \times_{(\phi_1,\phi_2)} G_2 \coloneqq 
\{(g_1, g_2) \in G_1 \times G_2 : \phi_1(g_1) = \phi_2(g_2)\}.
\]
This is a subgroup of $G_1 \times G_2$ that surjects onto both $G_1$ and $G_2$ via the usual projection maps.  
Goursat's lemma, stated below, gives that \emph{every} subgroup of $G_1 \times G_2$ surjecting onto both factors arises as a fiber product of $G_1$ and $G_2$.

\begin{lemma}[Goursat's Lemma]\label{goursat}  Let $G$, $G_1$, and $G_2$ be groups. If $G$ is a subgroup of $G_1 \times G_2$ that surjects onto $G_1$ and $G_2$ via the usual projection maps, then $G = G_1 \times_{(\phi_1,\phi_2)} G_2$ for some group $Q$ and surjective homomorphisms $\phi_1 \colon G_1 \to Q$ and $\phi_2 \colon G_2 \to Q$.
\end{lemma}
\begin{proof}
    See, \cite[p.\ 75]{MR1878556} or \cite[Sec.\ 1.2.2]{BrauThesis}, for instance.
\end{proof}

With these preliminaries, we turn to the product formula.

\begin{proposition} \label{P:6-p} Let $E/\Q$ be a non-CM elliptic curve that admits a rational isogeny of prime degree $p$. 
\begin{enumerate}
    \item If $p \in \{5, 7, 13\}$, then
    \[
    [\GL_2(\widehat{\Z}) : G_E] = \left[\SL_2(\Z_6) : G_E(6^\infty)'\right] \left[\SL_2(\Z_p) : G_E(p^\infty)' \right].
    \]
\item If $p \in \{2,3\}$ and the $\ell$-adic Galois representation of $E$ is surjective for all primes $\ell \geq 5$, then
    \[
    [\GL_2(\widehat{\Z}) : G_E] = \left[\SL_2(\Z_6) : G_E(6^\infty)'\right].
    \]
\end{enumerate}
\end{proposition}
\begin{proof} By  \cite[Lemma 1.7]{ZywinaOpen}, we have that
\begin{equation}\label{SL2indx}
      [\GL_2(\widehat{\Z}) : G_E] = [\SL_2(\widehat{\Z}) : G_E']. 
\end{equation}
 
Thus, it suffices to determine the commutator index of the adelic image. We now prove part (1) of the proposition. The proof of part (2) follows similarly.  Assume that $p \in \{5, 7, 13\}$. By Propositions \ref{P:5isog-elladic}, \ref{P:7isog-elladic}, and \ref{P:13isog-elladic}, we have that the $\ell$-adic Galois representation of $E$ is surjective for all primes $\ell \not\in \{2,3,p\}$. Let $m$ be the least common multiple of the level of $G_E(2^\infty\cdot3^\infty\cdot p^\infty)$ as a subgroup of $\GL_2(\widehat{\Z})$ and the level of $G_E(2^\infty \cdot 3^\infty \cdot p^\infty)'$ as a subgroup of $\SL_2(\widehat{\Z})$. Let $m_0$ be as in Theorem \ref{P:m0}. Taking $G$ to be the full preimage of $G_E(m)$ in $\GL_2(\widehat{\Z})$ and $H$ to be $G_E$, we see that the hypotheses of Theorem \ref{P:m0} are satisfied. Applying the theorem, we obtain
    \begin{equation} \label{E:ComProd2}
    G' = G_E'.
    \end{equation}
    Since the level of $G$ divides $m$, up to the natural isomorphism we have that
    \[
    G = G_m \times \GL_2(\Z_{(m)}) \subseteq \GL_2(\Z_m) \times \GL_2(\Z_{(m)})
    \]
    where $G_m$ denotes the image of $G$ in $\GL_2(\Z_m)$ and $\Z_{(m)} = \prod_{\ell \nmid m} \Z_\ell$. Thus,
    \begin{equation} \label{E:ComProd3}
    G' =(G_m \times \GL_2(\Z_{(m)}))' 
    = G_m' \times \GL_2(\Z_{(m)})'.
    \end{equation}
    Next, we observe from $G_m = G_E(m^\infty)$ that
    \begin{equation} \label{E:ComProd4}
        G_m' = G_E(m^\infty)'.
    \end{equation}
    Identifying $G_E(m^\infty)$ as a subgroup of $\GL_2(\Z_6) \times \GL_2(\Z_p)$, we have that
    \[
    G_E(m^\infty) \subseteq G_E(6^\infty) \times G_E(p^\infty).
    \]
    Thus
    \[
    G_E(m^\infty)' \subseteq G_E(6^\infty)' \times G_E(p^\infty)'.
    \]
    Furthermore, $G_E(m^\infty)'$ surjects onto both of the factors above by the natural projection maps. Thus, by Goursat's lemma (Lemma \ref{goursat}) we have that
    \begin{equation}
    G_E(m^\infty)'
    = G_E(6^\infty)' \times_{(\phi_1, \phi_2)} G_E(p^\infty)' ,\label{E:com-m6ell}
    \end{equation}
    where $\phi_1 \colon G_E(6^\infty)' \to Q$ and $\phi_2 \colon G_E(p^\infty)' \to Q$ are surjective group homomorphisms onto a group $Q$. By Lemma \ref{L:CommonQuo}, it must be that $Q = \{1\}$, so \eqref{E:com-m6ell} is actually a direct product. In symbols,
    \[
    G_E(m^\infty)'
    = G_E(6^\infty)' \times  G_E(p^\infty)'.
    \]
    Thus,
    \[
    [\SL_2(\Z_m) : G_E(m^\infty)']
    = [\SL_2(\Z_6) : G_E(6^\infty)'] [\SL_2(\Z_p) : G_E(p^\infty)'].
    \]
    Combining the above equation with  \eqref{SL2indx}, \eqref{E:ComProd2}, 
     \eqref{E:ComProd3}, and \eqref{E:ComProd4}, we obtain the desired result.
\end{proof}

\section{Handling the Primes \texorpdfstring{$p = 11, 17$ and $37$}{p = 11, 17, and 37}}\label{s:11and17and37}

The rational points of $X_0(p)$ for $p \in \{11,17,37\}$ are well known in the literature. See page 79 of \cite{MR376533} for $p \in \{11,17\}$ and Section 5, Proposition 2 of \cite{MR354674} for $p=37.$ In this section, we list all of the non-CM $j$-invariants that arise and verify the adelic index for each lies in $\mathcal{I}_p$. 
\begin{itemize}
    \item For $X_0(11)$ (\mc{11.12.1.a.1}), there are exactly five rational points, of which two are non-CM. Their $j$-invariants are $-11^2$ and $ -11 \cdot 131^{3}$. In both cases, the adelic index is $480$. 
    \item For $X_0(17)$ (\mc{17.18.1.a.1}), there are exactly four rational points, of which two are non-CM. Their $j$-invariants are $ -2^{-1} \cdot 17^{2} \cdot 101^{3}$ and $ -2^{-17} \cdot 17 \cdot 373^{3}$. In both cases, the adelic index is $576$.
    \item For $X_0(37)$ (\mc{37.38.2.a.1}), there are exactly four rational points, of which two are non-CM. Their $j$-invariants are $-7 \cdot 11^{3}$ and $-7 \cdot 137^{3} \cdot 2083^{3}$. In both cases, the adelic index is $2736$. 
\end{itemize}

In all three cases, the only adelic indices correspond to the sets $\mathcal{I}_{11},\mathcal{I}_{17}$, and $\mathcal{I}_{37}$ of Theorem \ref{t:mainthm}.

\section{Handling the Primes \texorpdfstring{$p = 5, 7,$ and $13$}{p = 5, 7, and 13}}\label{s:5and7and13}

\subsection{The Lattice Construction} \label{S:LatCons}

For the primes $p \in \{5, 7, 13\}$, we determine many $6$-adic images 
\[
G_E(6^\infty) \subseteq \GL_2(\Z_6)
\]
that contain $-I$ and occur for some non-CM elliptic curve $E/\Q$ admitting a rational $p$-isogeny. Together with Proposition \ref{P:6-p} and the results of Section \ref{s:possibleimages}, this will be sufficient for us to determine all adelic indices for non-CM elliptic curves over $\Q$ admitting a rational $p$-isogeny for each $p \in \{5,7,13\}$.

Our approach follows a lattice-based strategy, drawing inspiration from  the articles \cite{MR3500996,MR4468989,ZywinaOpen}. Fix a prime $p \in \{5, 7, 13\}$ and a subgroup $K \subseteq \GL_2(\Z_p)$ that contains $-I$ and arises as the $p$-adic Galois image for infinitely many $\overline{\Q}$-isomorphism classes of non-CM elliptic curves over $\Q$ admitting a rational $p$-isogeny. We construct a finite lattice $\mathcal{L}_p(K)$, consisting of certain open subgroups of $\GL_2(\Z_6)$, recursively as follows:

\begin{enumerate}
    \item The group $\GL_2(\Z_6)$ is in $\mathcal{L}_p(K)$. 
    \item If $H_0 \in \mathcal{L}_p(K)$ is such that the fiber product of modular curves $X_{H_0} \times X_K$ has infinitely many rational points, then each maximal open subgroup $H$ of $H_0$ satisfying all of the following conditions is also included in $\mathcal{L}_p(K)$:
    \begin{itemize}
        \item $-I \in H$ and $\det H = \Z_6^\times$,
        \item The image of $H$ under the projection $\GL_2(\Z_6) \to \GL_2(\Z_2)$ is the $2$-adic image (up to conjugacy) of some non-CM elliptic curve over $\Q$ admitting a rational $p$-isogeny, and
        \item The image of $H$ under the projection $\GL_2(\Z_6) \to \GL_2(\Z_3)$ is the $3$-adic image (up to conjugacy) of some non-CM elliptic curve over $\Q$ admitting a rational $p$-isogeny.
    \end{itemize}
\end{enumerate} 

To enumerate all maximal open subgroups of a given open subgroup of $\GL_2(\Z_6)$, we use Corollary \ref{L:Zyw-Lemma7.3}, which allows us to work at a finite (and computationally tractable) level. It is straightforward to check the conditions in the first bullet point. For the second and third bullet points, we use the results of Section \ref{s:possibleimages}:  Propositions \ref{P:5isog-elladic}, \ref{P:7isog-elladic}, and \ref{P:13isog-elladic} for $p=5, 7,$ and $13$ respectively.

Further, we define the sublattices
\begin{align*} 
\mathcal{L}_p^{\mathrm{fin}}(K) &\coloneqq \{ H \in \mathcal{L}_p(K) : |(X_H \times X_K)(\Q)| < \infty \} \\
\mathcal{L}_p^{\mathrm{inf}}(K) &\coloneqq \{ H \in \mathcal{L}_p(K) : | (X_H \times X_K)(\Q)| = \infty \}.
\end{align*}
We compute rational points on the modular curves attached to those groups in $\mathcal{L}_p^{\mathrm{fin}}(K)$ that are maximal within this sublattice  (i.e., not contained up to conjugacy in any other group of $\mathcal{L}_p^{\mathrm{fin}}(K)$). Determining the rational points on these modular curves gives all $6$-adic images containing $-I$ that occur for only finitely many $\overline{\Q}$-isomorphism classes of non-CM elliptic curves over $\Q$ with $p$-adic image contained in $K$ (up to conjugacy). Such $6$-adic images containing $-I$ that occur infinitely often are a subset of the groups in $\mathcal{L}_p^{\mathrm{inf}}(K)$ (perhaps a proper subset due to the phenomenon of ``curious groups'' \cite{chiloyan2023classification}, but this will not matter for our purposes).

The \texttt{Magma} code for constructing the lattices described above (for a given $p$ and $K$) appears in the file \texttt{Approach\textunderscore Lattice.m} in our GitHub repository. The files \texttt{Case p = 5.m}, \texttt{Case p = 7.m}, and \texttt{Case p = 13.m} call this code with the relevant prime $p$ and subgroups $K$ in each case.

\subsection{Rational Isogenies of Degree \texorpdfstring{$5$}{5}}\label{s:rationalisogeniesdegree5}

Let $p = 5$ and $K \subseteq \GL_2(\Z_5)$ be a subgroup containing $-I$ that appears as the $5$-adic image of
some non-CM elliptic curve over $\Q$ admitting a rational $5$-isogeny. By Proposition \ref{P:5isog-elladic}, there are eight possibilities for $K$ up to conjugacy. For each of these, we compute the commutator index in the \texttt{Magma} file \texttt{Case p = 5.m}, finding that
\[[\SL_2(\Z_5) : K'] \in \{24, 600\}.\] 
Because $K$ arises from an elliptic curve admitting a rational $5$-isogeny, $K$ is necessarily conjugate to a subgroup of \mc{5.6.0.a.1} ($X_0(5)$). In \texttt{Case p = 5.m}, we verify that if $[\SL_2(\Z_5) : K'] = 600$, then $K$ is conjugate to a subgroup of either \mc{5.30.0.a.1} ($X_{\mathrm{sp}}(5)$) or \mc{25.30.0.a.1} ($X_0(25)$). Thus, to determine the adelic index of any non-CM elliptic curve over $\Q$ admitting a rational $5$-isogeny, it suffices to determine the possible $6$-adic images that can occur along with a $5$-adic image of \mc{5.6.0.a.1} ($X_0(5)$), \mc{5.30.0.a.1} ($X_{\mathrm{sp}}(5)$), or \mc{25.30.0.a.1} ($X_0(25)$). We address each of these now.

\subsubsection{The lattice $\mathcal{L}_5(X_0(5))$}

We begin by studying the lattice $\mathcal{L}_5(X_0(5))$ associated with the group \mc{5.6.0.a.1} ($X_0(5)$). In the file \texttt{Case p = 5.m}, we construct the lattice $\mathcal{L}_5(X_0(5))$  described in Section \ref{S:LatCons} and depicted in Figure \ref{f:x0(5)}. The nodes in the diagram represent the groups contained in $\mathcal{L}_5(X_0(5))$. The white nodes are those in $\mathcal{L}_5^{\mathrm{inf}}(X_0(5))$, the gray nodes are maximal within $\mathcal{L}_5^{\mathrm{fin}}(X_0(5))$, and the black nodes are non-maximal within $\mathcal{L}_5^{\mathrm{fin}}(X_0(5))$. 

\begin{figure}[H] 
{\includegraphics[width=\textwidth]{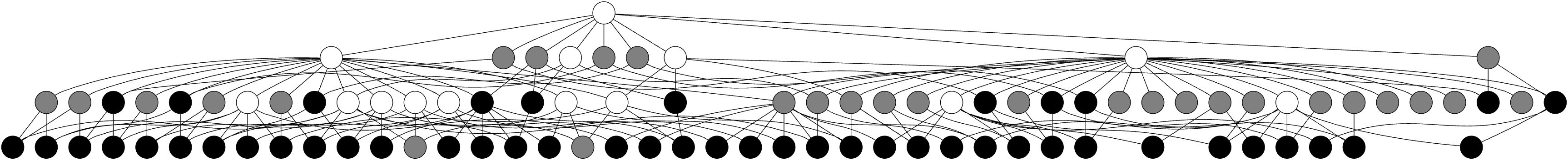}}\caption{The lattice $\mathcal{L}_5(X_0(5))$}\label{f:x0(5)}
\end{figure} 

We begin by considering the sublattice $\mathcal{L}_5^{inf}(X_0(5))$. We have that $|\mathcal{L}_5^{\mathrm{inf}}(X_0(5))| = 14$. By computing the commutator index for each of these $14$ groups, we find that
\[
\{ [\SL_2(\Z_6) : H'] : H \in \mathcal{L}_5^{\mathrm{inf}}(X_0(5)) \} = \{2, 12\}
\]
Therefore, in the case where the $5$-adic commutator index is $24$: If $E/\Q$ admits a rational $5$-isogeny and has a quadratic twist whose $6$-adic image lies in $\mathcal{L}_5^{\mathrm{inf}}(X_0(5))$ up to conjugacy, then 
\[ [\GL_2(\widehat{\Z}) : G_E] \in \{48, 288\} \subseteq \mathcal{I}_5. \]

It remains to consider the groups coming from the sublattice $\mathcal{L}_5^{\mathrm{fin}}(X_0(5))$. To do so, we need only determine the rational points on the modular curves associated with the $29$ gray (maximal finite) nodes. Of these, $23$ correspond to modular curves of genus $1$:
\begin{align*}
    &\mc{15.24.1.a.1}, \mc{15.36.1.b.1}, \mc{20.24.1.g.1}, \mc{30.12.1.d.1}, \mc{30.36.1.q.1}, \mc{30.36.1.r.1}, \\
    &\mc{60.12.1.j.1}, \mc{60.36.1.do.1},
    \mc{60.36.1.dp.1},
    \mc{60.36.1.dr.1},
    \mc{60.36.1.ga.1}, 
    \mc{60.36.1.gb.1},  \\
    & \mc{60.36.1.w.1},  \mc{60.36.1.x.1},
    \mc{120.12.1.bl.1},
    \mc{120.36.1.cw.1}, \mc{120.36.1.lm.1},
    \mc{120.36.1.ln.1},
      \\ &\mc{120.36.1.lo.1},
      \mc{120.36.1.lp.1}, 
      \mc{120.36.1.lq.1},
      \mc{120.36.1.sm.1},
         \text{ and }\mc{120.36.1.so.1}. 
\end{align*}

 Each of these curves either has no rational points or is an elliptic curve of rank 0. We compute all rational points and determine which are non-CM using the LMFDB. For each of these, we check the adelic index of an elliptic curve over $\Q$ with the corresponding $j$-invariant. These are listed in Table \ref{tab:g1jinvs}. Doing so, we find that all adelic indices coming from these genus 1 modular curves are among 
 \begin{equation*}
 \{48,192,288,384,576,768\}\subseteq\mathcal{I}_5.
 \end{equation*}

\begin{table}[H]
    \centering
    \begin{tabular}{|c|c|}
    \hline
        \textbf{LMFDB Label} & \textbf{Non-CM $j$-invariants} \\ \hline 
         \mc{15.24.1.a.1}  & $\frac{-5^2}{2}, \frac{-5\cdot29^3}{2^5}, \frac{5\cdot211^3}{2^{15}}, \frac{-5^2\cdot241^3}{2^3}$ \\ \hline
\mc{15.36.1.b.1} & $\frac{11^3}{2^3},\frac{-29^3\cdot41^3}{2^{15}}$  \\ \hline
\mc{20.24.1.g.1}  & $\frac{5\cdot59^3}{2^{10}},\frac{-5^2\cdot41^3}{2^2}$  \\ \hline
\mc{30.12.1.d.1} & $\frac{2^{12}\cdot5}{3^5},\frac{-2^{12}\cdot5^2}{3}$ \\ \hline
\mc{30.36.1.q.1} & $2^{12},2^{12}\cdot211^3$  \\ \hline
    \end{tabular}
    \caption{Non-CM $j$-invariants for rank 0 elliptic curves in $\mathcal{L}_5^{\mathrm{fin}}(X_0(5))$}
    \label{tab:g1jinvs}
\end{table}

Two of the curves are of genus $2$: \mc{30.54.2.a.1} and \mc{45.54.2.c.1}. Their  Jacobians have rank $0$ and $1$, respectively. Using the \texttt{Chabauty0} command in \texttt{Magma}, we find that \mc{30.54.2.a.1} has exactly four rational points, all of which are cusps. For \mc{45.54.2.c.1}, we find a degree $0$ divisor of infinite order and then use the \texttt{Chabauty} command to conclude that there are exactly four rational points, among which two are non-CM. Their $j$-invariants are $\{2^6,-2^6\cdot719^3\}$. The adelic index of both of these points is $864$, which lies in $\mathcal{I}_5$.

Finally, the remaining four curves are of genus $3$: \mc{120.36.3.c.1}, \mc{120.36.3.a.1}, \mc{120.72.3.bee.1}, and \mc{120.72.3.ka.1}. The first two cover \mc{24.6.1.c.1} and \mc{24.6.1.a.1}, respectively, which are elliptic curves of rank $0$. Each of these elliptic curves has exactly two rational points: one is a cusp and the other is CM. This means there are no non-CM rational points on  \mc{120.36.3.c.1} and \mc{120.36.3.a.1}. The latter two curves map to the bielliptic curve given by the equation \[y^2+y=-2x^6+9x^2-7\] which, by Table \ref{tab:biellipticchabauty}, has four rational points. On both \mc{120.72.3.bee.1} and \mc{120.72.3.ka.1}, these four points lift to eight rational points. The $j$-invariants of the non-CM rational points of \mc{120.72.3.bee.1} are $\{\frac{2^6\cdot11^3}{3}, \frac{2^6\cdot971^3}{3^5}\}$. The $j$-invariants of the non-CM rational points of \mc{120.72.3.ka.1} are $\frac{64}{9}$ and $\frac{-2^6\cdot239^3}{3^{10}}$. The adelic index of any elliptic  curve over $\Q$ with one of these $j$-invariants is $288$, which appears in $\mathcal{I}_5$. This completes our analysis of $\mathcal{L}_5(X_0(5))$.

\subsubsection{The lattice $\mathcal{L}_5(X_{sp}(5))$}
In contrast to the case of $X_0(5)$, the lattice associated with $X_{\mathrm{sp}}(5)$ is rather simple. We construct this lattice in the file \texttt{Case p = 5.m}.

\begin{figure}[H] 
\centering
{\includegraphics[height=.65in]{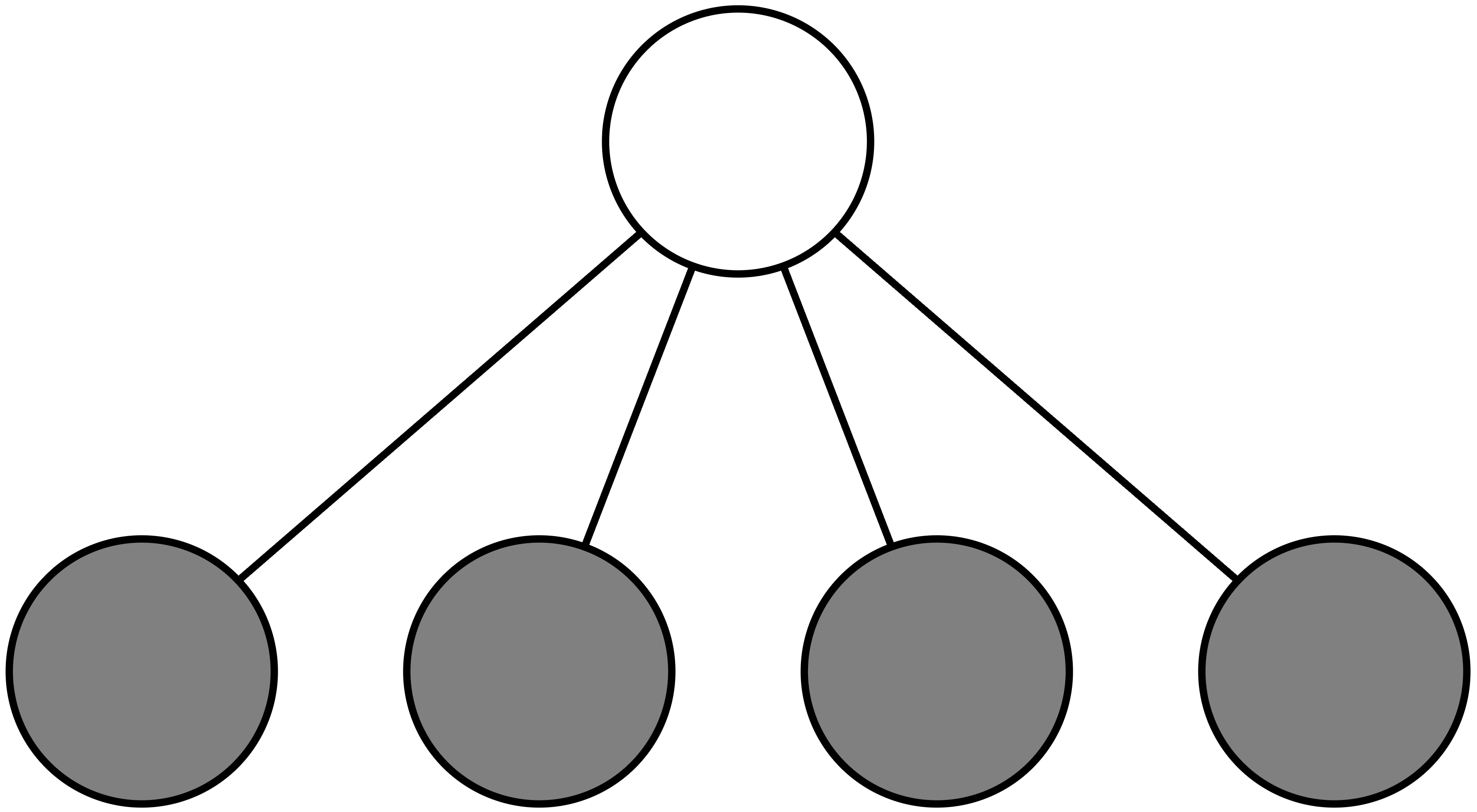}}\caption{The lattice $\mathcal{L}_5(X_{\mathrm{sp}}(5))$}
\end{figure}

Again, we begin by considering the one node that constitutes $\mathcal{L}_5^{\mathrm{inf}}$. We find that $\mathcal{L}_5^{\mathrm{inf}}(X_{\mathrm{sp}}(5)) = \{\GL_2(\Z_6)\}$. Thus, the only adelic index arising from the sublattice $\mathcal{L}_5^{\mathrm{inf}}(X_{\mathrm{sp}}(5))$ by Proposition \ref{P:6-p} is
\[
[\GL_2(\widehat{\Z}) : G_E] = 2 \cdot 600 = 1200 \in \mathcal{I}_5.
\]

We now consider the four groups in $\mathcal{L}_5^{\mathrm{fin}}(X_{\mathrm{sp}}(5))$.  The associated modular curves are \mc{30.60.3.c.1}, \mc{60.60.3.h.1}, \mc{120.60.3.z.1}, and \mc{120.60.3.bf.1}.

The modular curve \mc{30.60.3.c.1} covers \mc{30.12.1.d.1} which is an elliptic curve of rank 0 and has exactly two non-CM $j$-invariants: $\{\tfrac{2^{12}\cdot5}{3^5},\tfrac{-2^{12}\cdot5^2}{3}\}$. We check that if an elliptic curve over $\Q$ has one of these $j$-invariants, then the index of its $5$-adic image must be either 24 or 12. The 5-adic image cannot be a subgroup of $X_{\mathrm{sp}}(5)$. Hence, there are no non-CM rational points of \mc{30.60.3.c.1}. 

The modular curve \mc{60.60.3.h.1} covers \mc{60.12.1.j.1}, which is an elliptic curve of rank 0 and has exactly two rational points. Both are cusps, so there are no non-CM rational points of \mc{60.60.3.h.1}.

The modular curve \mc{120.60.3.z.1} is the fiber product of \mc{24.2.0.a.1} and \mc{5.30.0.a.1}. It is a genus $3$ curve that covers the modular curve \mc{120.30.2.e.1}, which has a hyperelliptic model 
\[ y^2 = -12x^5 - 18x^4 + 63x^3 - 18x^2 - 12x.\]
This curve is bielliptic of genus $2$ and rank $2$, so we can compute its rational points using quadratic Chabauty and the Mordell--Weil sieve (see Section \ref{s:quadraticchabauty}). This yields one cusp and one CM point.

The modular curve \mc{120.60.3.bf.1} is the fiber product of $X_{\mathrm{sp}}(5)$ and \mc{24.2.0.b.1}. Immediately, it covers the fiber product of $X_0(5)$ and \mc{24.2.0.b.1}, which is labeled \mc{120.12.1.bl.1} and is an elliptic curve of rank $0$. It has only two rational points, both of which are cusps. Hence, the only rational points of \mc{120.60.3.bf.1} are cusps.

\subsubsection{The lattice $\mathcal{L}_5(X_0(25))$} Finally, we consider the lattice associated with $X_{0}(25)$. It has the same tree structure as in the previous case and is shown in Figure \ref{tab:Xsp5}. Again, this lattice is constructed in the file \texttt{Case p = 5.m}. 

\begin{figure}[H] 
\centering
{\includegraphics[height=.65in]{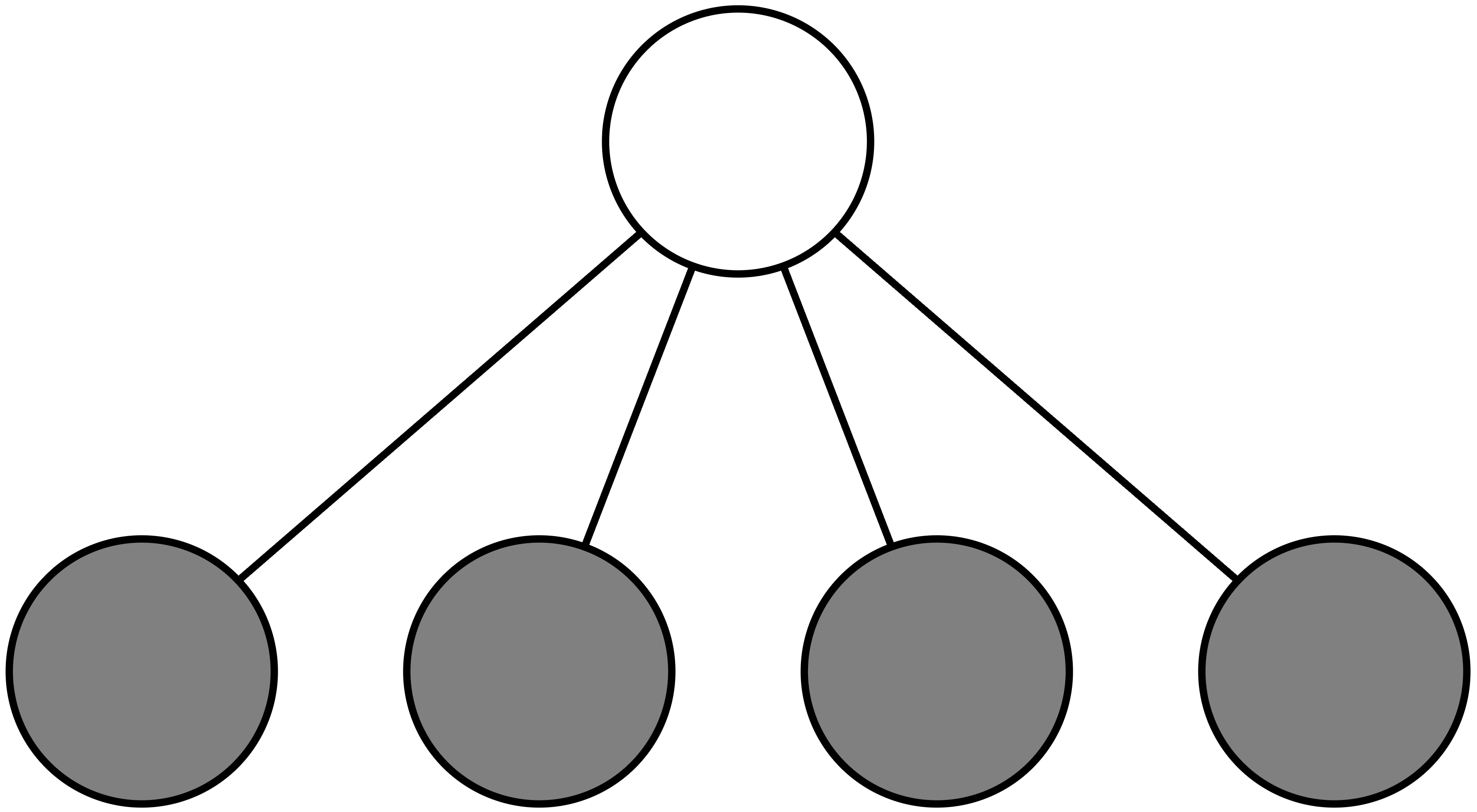}}\caption{The lattice $\mathcal{L}_5(X_0(25))$}\label{tab:Xsp5}
\end{figure} 

Since $\mathcal{L}_5^{\mathrm{inf}}(X_{0}(25)) = \{\GL_2(\Z_6)\}$, the only adelic index coming from the sublattice $\mathcal{L}_5^{\mathrm{inf}}(X_{0}(25))$ is
\[
[\GL_2(\widehat{\Z}) : G_E] = 2 \cdot 600 = 1200 \in \mathcal{I}_5.
\]

We now consider the four groups in $\mathcal{L}_5^{\mathrm{fin}}(X_{0}(25))$, all of which are maximal, as seen in the lattice diagram. The associated modular curves are  \mc{150.60.3.a.1}, \mc{300.60.3.a.1}, \mc{24.2.0.a.1} $\times$ \mc{25.30.0.a.1}, and \mc{24.2.0.b.1} $\times$ \mc{25.30.0.a.1}. Note that the latter two curves do not yet appear in the LMFDB Beta, so we instead write them as products.

The modular curve \mc{150.60.3.a.1} is the fiber product of $X_0(25)$ and \mc{6.2.0.a.1}. Suppose, toward a contradiction, there exists a non-CM elliptic curve $E/\Q$ that corresponds to a rational point of \mc{150.60.3.a.1}. From Lemma \ref{lemma:isomorphicimages}, if $E'/\Q$ is such that there exists a $5$-isogeny between $E$ and $E'$, then $G_E(6)$ is conjugate to $G_{E'}(6)$. Since $E$ admits a $25$-isogeny, there must exist some non-CM elliptic curve $E'/\Q$ which is $5$-isogenous to $E$ that gives rise to a rational point on the fiber product of $X_{\mathrm{sp}}(5)$ and \mc{6.2.0.a.1}. This fiber product is \mc{30.60.3.c.1}. We have already concluded that there are no non-CM rational points of \mc{30.60.3.c.1}, so there cannot be any non-CM rational points of \mc{150.60.3.a.1}. A similar argument applies for the other three modular curves as well. With this, we have shown that there are no non-CM rational points on any of these four curves. 

This completes the proof of Theorem \ref{t:mainthm} in the case of $p = 5$. All of the rational point computations are verified in the file \texttt{Section 7.2.m} in the repository \cite{projectcode}.

\subsection{Rational Isogenies of Degree \texorpdfstring{$7$}{7}}\label{s:rationalisogeniesdegree7}

Next, we consider the case  $p=7$. The four $7$-adic images from Proposition \ref{P:7isog-elladic} each have a commutator index of $48$ inside $\SL_2(\Z_7)$. Thus, unlike the case of $p = 5$, we need only consider a single lattice, namely the lattice associated with $X_0(7)$.  

\begin{figure}[H] 
\centering
{\includegraphics[height=1.2in]{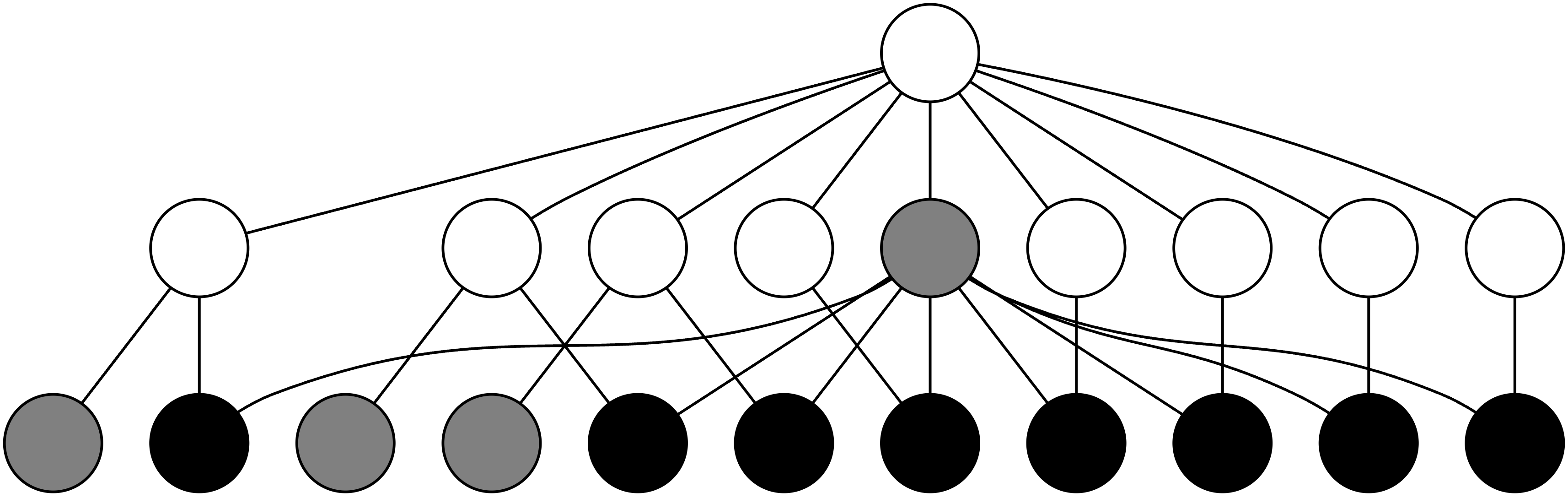}}\caption{The lattice $\mathcal{L}_7(X_0(7))$}
\end{figure}

The coloring scheme (perhaps it would be more accurate to say ``shading scheme'') of the nodes is the same as in the previous section. We have that $|\mathcal{L}_7^{\mathrm{inf}}(X_0(7))| = 9$. By computing the commutator index for each of these groups, we find that
\[
\{ [\SL_2(\Z_6) : H'] : H \in \mathcal{L}_7^{\mathrm{inf}}(X_0(7)) \} = \{2, 4, 12\}.
\]
Thus, the possible adelic indices associated with these nodes in the sublattice $\mathcal{L}_7^{\mathrm{inf}}(X_0(7))$ are
\[ [\GL_2(\widehat{\Z}) : G_E] \in \{96, 192, 576\} \subseteq \mathcal{I}_7. \]

It remains to study the modular curves associated with the four gray nodes; these are maximal within $\mathcal{L}_7^{\mathrm{fin}}(X_0(7))$. Their labels are \mc{21.32.1.a.1}, \mc{28.64.3.b.1}, \mc{42.48.2.d.1}, and \mc{126.48.2.a.1}. 

The rational points on \mc{21.32.1.a.1}, $X_0(21)$, are well known. They give four non-CM $j$-invariants:
\[ \{\tfrac{3^3\cdot5^3}{2}, \tfrac{-3^2\cdot5^6}{2^3}, \tfrac{-3^2\cdot5^3\cdot101^3}{2^{21}}, \tfrac{-3^3\cdot5^3\cdot383^3}{2^7}\}.\]
The adelic index of any elliptic curve over $\Q$ with any of these $j$-invariants is $768 \in \mathcal{I}_7$. The curve \mc{28.64.3.b.1} has genus $3$ and maps to a rank $0$ elliptic curve. Using the code \texttt{ModCrvToEC
}\cite{JacobJeremycode}, we determine that it has exactly eight rational points, of which four are non-CM. Their $j$-invariants are in $\{\tfrac{3^3\cdot13}{2^2}, \tfrac{-3^3\cdot13\cdot479^3}{16384}\}$ and both give adelic index $768$. The curves \mc{42.48.2.d.1} and \mc{126.48.2.a.1} are both genus $2$ rank $0$, so we use the command \texttt{Chabauty0} in \texttt{Magma} to determine their rational points. The curve \mc{42.48.2.d.1} has $2$ rational points, both of which are cusps. Lastly, the curve \mc{126.48.2.a.1} has no rational points.

This completes the proof of Theorem \ref{t:mainthm} in the case of $p=7$. All of the rational point computations are verified in the file \texttt{Section 7.3.m} in the repository \cite{projectcode}.

\subsection{Rational Isogenies of Degree \texorpdfstring{$13$}{13}}\label{s:rationalisogeniesdegree13}

Lastly, we consider the case $p = 13$. The six $13$-adic images from Proposition \ref{P:13isog-elladic} each have a commutator index of $168$ inside $\SL_2(\Z_{13})$. Thus, we again need only consider a single lattice, associated with $X_0(13)$.

\begin{figure}[H]
\centering
{\includegraphics[height=1.2in]{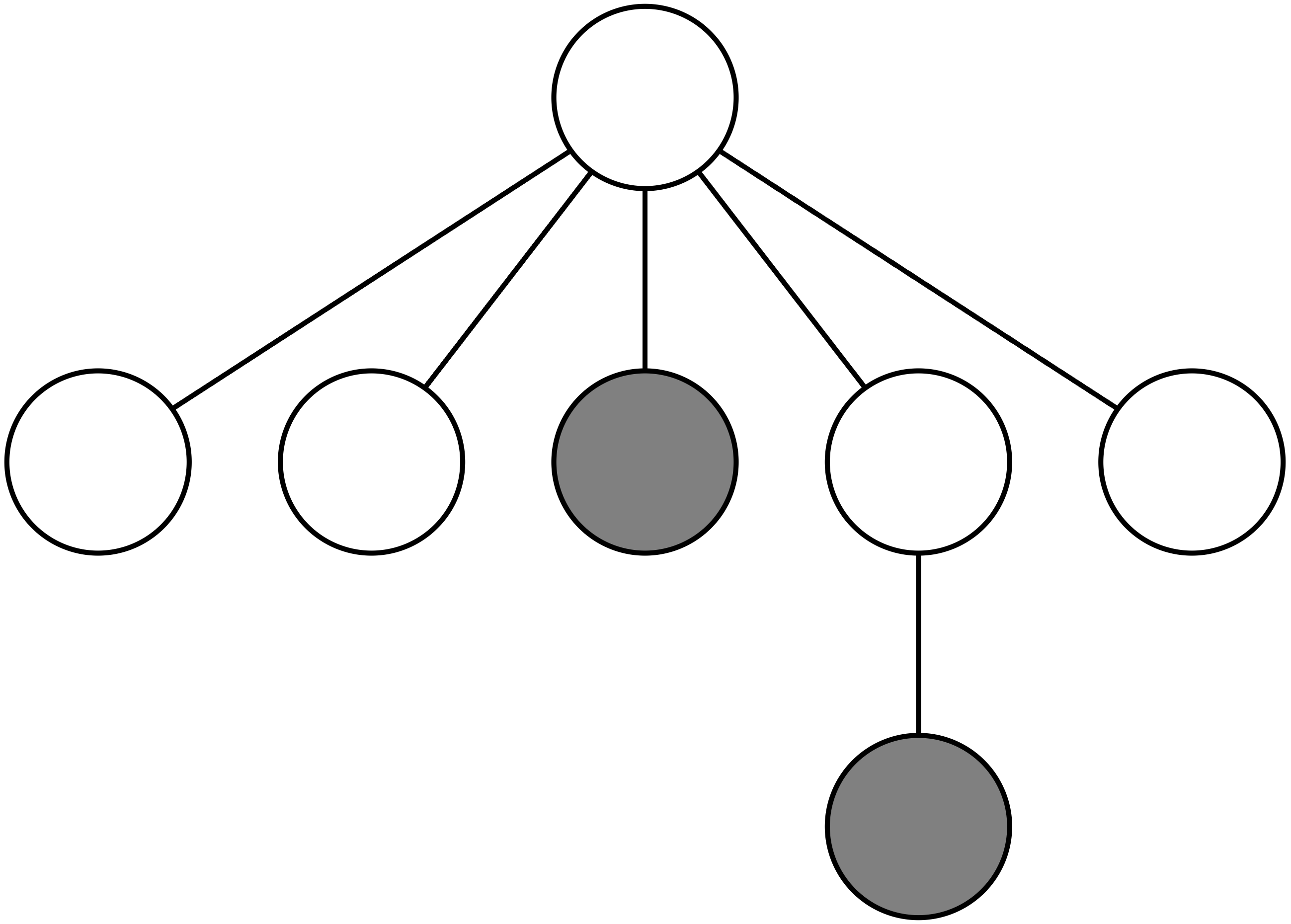}}
\caption{The lattice $\mathcal{L}_{13}(X_0(13))$}
\label{f:x013}
\end{figure}

The coloring scheme of the nodes is the same as in the previous section. We have that $|\mathcal{L}_{13}^{\mathrm{inf}}(X_0(13))| = 5$. By computing the commutator index for each of these groups, we find that
\[
\{ [\SL_2(\Z_6) : H'] : H \in \mathcal{L}_{13}^{\mathrm{inf}}(X_0(13)) \} = \{2\}.
\]
Thus, the only possible adelic index associated with these nodes in the lattice is
\[ [\GL_2(\widehat{\Z}) : G_E] = 336 \in \mathcal{I}_{13}. \]

It remains to study the modular curves associated with the two gray nodes. These are \mc{78.84.5.b.1} and \mc{312.28.1.a.1}. The first curve admits a map to a rank $0$ elliptic curve. Using the code \texttt{ModCrvToEC
}\cite{JacobJeremycode}, we determine that there are exactly two rational points on this curve, both of which are cusps. The second curve, \mc{312.28.1.a.1}, also has rank $0$ and two rational points, again both of which are cusps.

This completes the proof of Theorem \ref{t:mainthm} in the case of $p=13$. All of the rational point computations are verified in the file \texttt{Section 7.4.m} in the repository \cite{projectcode}.

\section{Handling the Primes \texorpdfstring{$p = 2$ and $3$}{p = 2 and 3}}\label{s:2and3}

\subsection{The Product Approach}\label{s:productapproach} For $p \in \{2,3\}$, we take a product-based strategy instead of the lattice-based approach from the previous section. Let $E/\Q$ be a non-CM elliptic curve that admits a rational isogeny of degree 2 or 3, under the additional hypothesis that it does not admit any rational isogenies of larger prime degree, consistent with the assumptions in Propositions \ref{P:2isog-elladic} and \ref{P:3isog-elladic}. By Propositions \ref{P:2isog-elladic} and \ref{P:3isog-elladic}, we know all possible $2$-adic and $3$-adic Galois images of $E$ that contain $-I$. We denote these sets of images by $\mathcal{L}_2 \coloneqq \mathcal{L}_2(p)$ and $\mathcal{L}_3 \coloneqq \mathcal{L}_3(p)$, respectively. Assuming the $\ell$-adic Galois representation is surjective for all $\ell \geq 5$, by Proposition \ref{P:6-p},
the adelic index of $E$ depends only on its $6$-adic image. More specifically, in this case we have   \[
    [\GL_2(\widehat{\Z}) : G_E] = \left[\SL_2(\Z_2) : G_E(2^\infty)' \right] \left[\SL_2(\Z_3) : G_E(3^\infty)' \right] \delta_E,
    \]
    where $\delta_E$ is the \emph{entanglement factor}, defined as the order of the common quotient group $Q$ of $G_E(2^\infty)'$ and $G_E(3^\infty)'$, as in Goursat's lemma. Proposition \ref{P:2isog-elladic} guarantees that the $\ell$-adic representation is surjective for all $\ell \geq 5$ in the $p = 2$ case; in the $p = 3$ case, it is possible that the $5$-adic representation is nonsurjective, and we deal with this possibility at the beginning of Section \ref{subsec:3isoglattice}.

For each direct product $G = H \times K$ with $H \in \mathcal{L}_2$ and $K \in \mathcal{L}_3$, we check whether the modular curve $X_G$ has finitely many rational points. Among those that do, we take all maximal such groups and determine all of their rational points (or explain why doing so is unnecessary for our purposes). We then proceed to the curves with infinitely many rational points. By iterating over all possible $2$-adic images, $3$-adic images, and entanglement factors, one obtains a superset of the possible adelic indices for the elliptic curves under consideration. Finally, we show that any resulting candidate  indices lying outside of $\mathcal{I}_p$ do not actually occur.

 \subsection{Rational Isogenies of Degree \texorpdfstring{$2$}{2}}\label{subsec:2isoglattice}

Let $E/\Q$ be a non-CM elliptic curve that admits a rational isogeny of degree $2$ and no rational isogenies of prime degree greater than $2$. Let $\mathcal{L}_2$ and $\mathcal{L}_3$ be as in the previous section. By Proposition \ref{P:2isog-elladic}, $\mathcal{L}_2$ consists of $188$ groups and  $\mathcal{L}_3 = \{ K_1, K_2, K_3 \}$ where
\[
K_1 = \GL_2(\Z_3), \quad K_2 = \mc{3.3.0.a.1}, \quad \text{and} \quad K_3 = \mc{3.6.0.b.1}.
\]
Further, the $\ell$-adic Galois representation of $E$ is surjective for all primes $\ell \geq 5$. 

\subsubsection{Maximal Products with Finitely Many Points}
Consider the collection of all direct products of the form $G = H \times K$ with $H \in \mathcal{L}_2$ and  $K \in \mathcal{L}_3$. Let $\mathcal{M}$ be the subset of such products $G$ that are maximal within this collection with respect to the property that the associated modular curve $X_G$ has only finitely many rational points. Note that the associated modular curves of their parent groups have infinitely many rational points. The set $\mathcal{M}$ is computed using the function \texttt{DetermineMaxProdsWithFinManyPts} in the file \texttt{Approach\textunderscore Product.m} and the output can be viewed by using \texttt{IdentifyProds(MaxProdsWithFinManyPts)} in the file \texttt{Case p = 2.m}. Running these commands, we find that $\mathcal{M}$ consists of $34$ groups, corresponding to the modular curves:

\begin{multicols}{4}
\begin{itemize}
    \item \mc{6.18.1.a.1}
    \item \mc{12.18.1.a.1}
    \item \mc{12.18.1.b.1}
    \item \mc{12.18.1.c.1}
    \item \mc{12.36.1.bo.1}
    \item \mc{12.36.1.bq.1}
    \item \mc{12.36.1.bs.1}
    \item \mc{12.36.1.bt.1}
    \item \mc{16.96.3.ey.1}
    \item \mc{16.96.3.ey.2}
    \item \mc{16.96.3.fa.1}
    \item \mc{16.96.3.fa.2}
    \item \mc{24.18.1.a.1}
    \item \mc{24.18.1.d.1}
    \item \mc{24.36.1.fc.1}
    \item \mc{24.36.1.fi.1}
    \item \mc{24.36.1.fo.1}
    \item \mc{24.36.1.fp.1}
    \item \mc{24.36.1.fx.1}
    \item \mc{24.36.1.gk.1}
    \item \mc{24.36.2.ba.1}
    \item \mc{24.72.2.hp.1}
    \item \mc{24.72.2.hp.2}
    \item \mc{24.72.2.jo.1}
    \item \mc{32.96.3.bh.1}
    \item \mc{32.96.3.bs.1}
    \item \mc{48.72.3.a.1}
    \item \mc{48.72.3.a.2}
    \item \mc{48.72.3.h.1}
    \item \mc{48.72.3.h.2}
    \item \mc{48.72.3.i.1}
    \item \mc{48.72.3.i.2}
    \item \mc{48.72.3.p.1}
    \item \mc{48.72.3.p.2}
\end{itemize}
\end{multicols}

Most of the rational points computations for these curves are straightforward using techniques already discussed earlier in this paper. For these, we provide full details in the file \texttt{Section 8.2.m}. We now discuss the computation for the curves that require more ad hoc techniques.

We tackle the rational points of \mc{48.72.3.h.1} and \mc{48.72.3.i.2} using the technique discussed in Section \ref{subsection:sieving}. First, the curve \mc{48.72.3.h.1}, which we denote by $C$, is genus 3 rank 1. It is defined in $\mathbb{P}^2$ by the equation \[x^3y + y^4 + 4y^2z^2 + 2z^4 =0.\]It has an automorphism $i \colon [x: y: z]  \mapsto [-x: -y: z]$. The quotient is a
rank 1 elliptic curve. Therefore, there exists an abelian variety of rank $0$, which we denote $V$, such that the Jacobian of this curve decomposes as $E \times V$. For any rational point $P$ on \mc{48.72.3.h.1}, the point $P-i(P)$ lies in the torsion subgroup of its Jacobian, which we denote by $J_C(\Q)_{\tors}$. It thus suffices to compute preimages of rational points in $J_C(\Q)_{\tors}$ under the map $a \colon C(\Q) \to J_C(\Q)_{\tors}$ given by $P \mapsto P-i(P)$. Note that $a$ is injective away from the fixed points of $i$. By using local computations, we can deduce that $J_C(\Q)_{\tors}$ is a subgroup of $\Z/2\Z \times \Z/2\Z$. Further, by working modulo $5$, we check that there is no divisor of the form $P - Q$ that has order $2$. Therefore, the only rational points of $C$ are the fixed points of this automorphism. It is easy to check that there are exactly two fixed points: $(1:0:0)$ and $(-1:1:0)$. One is a cusp and one is a CM point. Next, the curve \mc{48.72.3.i.2} is also of genus 3 and rank 1. It is defined in $\mathbb{P}^2$ by the equation \[2x^4 + 4x^2y^2 - 2xz^3 + y^4=0.\] It has an automorphism $i \colon (x: y: z)  \mapsto (x: -y: z)$. The quotient is a
rank 1 elliptic curve. By a similar analysis, we show that the only rational points of this curve are the fixed points of this automorphism. As before, we find exactly two fixed points (one cusp, one CM). For the curve \mc{48.72.3.i.1}(=\mc{16.24.0.m.1} $\times$ $X_{\mathrm{ns}}^{+}(3)$), it suffices to compute all the adelic indices of non-CM rational points of \mc{48.72.3.h.1}(=\mc{16.24.0.l.2} $\times$ $X_{\mathrm{ns}}^{+}(3)$), and we just determined that there are no such points. This suffices because, from Tables 11 and 12 of \cite{MR4868206}, if there is a non-CM $E/\Q$ whose $2$-adic image is a subgroup of \mc{16.24.0.m.1}, then there exists an $E'/\Q$ that is 2-isogenous to $E$ whose $2$-adic image will be a subgroup of \mc{16.24.0.l.2}. From Lemma \ref{lemma:isomorphicimages}, admitting a 2-isogeny does not affect the $3$-adic image. So, if there is a non-CM $E/\Q$ that corresponds to a rational point on \mc{48.72.3.i.1}, then there exists an $E'/\Q$ that corresponds to a rational point on \mc{48.72.3.h.1}. Further, from Corollary 4 of \cite{MR4830946}, we know that isogenous elliptic curves share the same adelic index. By a similar argument, for the curve \mc{48.72.3.h.2}, it suffices to compute all the adelic indices of non-CM rational points of \mc{48.72.3.i.2}, and we have already determined that there are no such points.

The modular curves \mc{16.96.3.ey.1}, \mc{16.96.3.ey.2}, \mc{16.96.3.fa.1}, and \mc{16.96.3.fa.2} are discussed in Section 9.2 of \cite{MR3500996}. Their noncuspidal, non-CM rational points along with their $j$-invariants are given. Each of the four curves has one non-CM $j$-invariant. These are $2^4\cdot17^3$, $2^{11}$, $\frac{17^3\cdot241^3}{2^4}$, and $\frac{257^3}{2^8}$, respectively. Each of their adelic indices is 384, which lies in the set $\mathcal{I}_2$ in Theorem \ref{t:mainthm}. The curves \mc{32.96.3.bh.1} and \mc{32.96.3.bs.1} are discussed in Section 9.4 of \cite{MR3500996}, and their noncuspidal, non-CM rational points along with their $j$-invariants are given. Each has one non-CM $j$-invariant. These are  $\tfrac{-3^3\cdot5^3\cdot47^3\cdot1217^3}{2^8\cdot31^8}$ and $\tfrac{3^3\cdot5^6\cdot13^3\cdot23^3\cdot41^3}{2^{16}\cdot31^4}$, respectively. The adelic index for these $j$-invariants is also 384.   

 With this, the remaining four curves to check are \mc{48.72.3.a.1} (= \mc{16.24.0.k.2} $\times \; K_2$), \mc{48.72.3.a.2} (= \mc{16.24.0.k.1} $\times \; K_2$), \mc{48.72.3.p.1} (=\mc{16.24.0.n.1} $\times$ $X_{\mathrm{ns}}^{+}(3)$) and \mc{48.72.3.p.2}. The first two curves have genus $3$ and rank $3$ and did not appear to be amenable to the techniques used in this article. Instead of computing the rational points on these two curves, we handle these group theoretically later in this section. We discuss \mc{48.72.3.p.1} and \mc{48.72.3.p.2} after that.

\subsubsection{Products with Infinitely Many Rational Points}
Next, we consider the direct products $G = H \times K$ with $H \in \mathcal{L}_2$ and $K \in \mathcal{L}_3$ for which $X_G$ has infinitely many rational points. By using the function \texttt{FindDeltaE(ProdsWithInfManyPts)} in the file 
\texttt{Case p = 2.m}, we find that 
\[
\delta_E \in
\begin{cases}
    \{ 1 \} & K = K_1 \\
    \{ 1, 2, 4 \} & K = K_2 \\
    \{ 1, 2\} & K = K_3.
\end{cases}
\]

By taking the product of commutator indices of $2$-adic images, $3$-adic images, and entanglement factors, we obtain the following set of possible indices: 
\begin{equation*}
    \{{12, 48, 72, 96, 144, 192, 288, 384, 576, 768, 864, 1152, 1536, 2304, 4608}\}.
\end{equation*}
All of these except $2304$ and $4608$  appear in Zywina's conjecture. Therefore, our goal for the remainder of this section is to rule out these two indices. 

Of the $202$ such groups $G$ with $X_G$ containing infinitely many rational points, there are a total of seven for which a fiber product could possibly give rise to an adelic index of $2304$ or $4608$. These are:

\begin{itemize}
\item For \mc{24.36.0.cg.1}, we must rule out both $2304$ and $4608$.
\item For \mc{24.36.0.cj.1}, we must rule out only $2304$.
\item For \mc{24.36.1.gl.1}, we must rule out only $2304$.
\item For \mc{48.72.0.c.1}, we must rule out both $2304$ and $4608$.
\item For \mc{48.72.0.c.2}, we must rule out both $2304$ and $4608$.
\item For \mc{48.72.0.d.1}, we must rule out only $2304$.
\item For \mc{48.72.0.d.2}, we must rule out only $2304$.
\end{itemize}

For \mc{24.36.0.cg.1}, \mc{24.36.0.cj.1},  and \mc{24.36.1.gl.1}, we consider all of the subgroups that are fiber products of the two respective factors that contain $-I$, have full determinant, and have commutator index outside $\mathcal{I}_{2}$. For subgroups that have commutator index outside $\mathcal{I}_2$, we compute the rational points of the corresponding modular curve.

For \mc{48.72.0.c.1}, \mc{48.72.0.c.2}, \mc{48.72.0.d.1}, and  \mc{48.72.0.d.2}, instead of considering all of the subgroups that are fiber products of two respective factors, contain $-I$, have full determinant, and have commutator index in \{2304, 4608\}, we only need to consider maximal subgroups that could potentially contain a subgroup that would lead to an index outside of $\mathcal{I}_2$.

Under this framework, each of the seven groups gives a list of modular curves we need to consider. These are listed in the following table:

{\small
\begin{longtable}{|c|l|}
\caption{Modular curves to consider to rule out 2304, 4608}\\
\hline
\textbf{$G$} & \textbf{Curves to consider } \\
\hline \endfirsthead
\multicolumn{2}{c}
{{\bfseries \tablename\ \thetable{} -- continued from previous page}} \\
\hline $G$ & \textbf{Curves to consider}\\ \hline 
\endhead
\hline \multicolumn{2}{|r|}{\textit{Continued on next page}} \\ \hline
\endfoot

\hline
\endlastfoot
    \mc{24.36.0.cg.1} & \makecell[l]{\mc{24.144.9.dga.1},
\mc{24.144.9.dgq.1},
\mc{24.144.9.eid.1},
\mc{24.144.9.ejb.1},\\
\mc{24.288.17.gil.1},
\mc{24.288.17.gim.1},
\mc{24.288.17.gim.2},
\mc{24.288.17.gin.1},\\
\mc{48.144.7.bfg.1},
\mc{48.144.7.bft.1},
\mc{48.144.7.ze.1},
\mc{48.144.7.zv.1},\\
\mc{48.144.9.bep.1},
\mc{48.144.9.bfa.1},
\mc{48.144.9.bjz.1},
\mc{48.144.9.bks.1},\\
\mc{48.144.11.sm.1},
\mc{48.144.11.sr.1},
\mc{48.288.21.gvf.1},
\mc{48.288.21.gvg.1},\\
\mc{48.288.21.gvg.2},
\mc{48.288.21.gvh.1},
} \\ \hline

\mc{24.36.0.cj.1} & \makecell[l]{\mc{48.288.15.cva.1}, 
\mc{48.288.15.cvb.1},
\mc{48.288.15.cvz.1},
\mc{48.288.15.cwc.1},\\
\mc{48.288.15.cyo.1},
\mc{48.288.15.cyp.1},
\mc{48.288.15.cyv.1},
\mc{48.288.15.cyy.1}}
\\ \hline

\mc{24.36.1.gl.1} & \makecell[l]{
\mc{48.288.15.cdy.1}, \mc{48.288.15.ceb.1},
\mc{48.288.17.fay.1}, 
\mc{48.288.17.faz.1},\\
\mc{48.288.19.hiq.1}, \mc{48.288.19.hit.1},
\mc{48.288.21.ezh.1}, \mc{48.288.21.ezi.1}} \\ \hline

\mc{48.72.0.c.1} & \makecell[l]{
\mc{48.144.3.bb.1}, \mc{48.144.3.r.1}, \mc{48.144.3.t.2},
\mc{48.144.3.z.1},
\mc{48.144.5.ee.2},\\ 
\mc{48.144.5.ef.1},
\mc{48.144.5.eg.1},
\mc{48.144.5.eh.2}, 
\mc{48.144.5.ei.2}, 
\mc{48.144.5.ej.1}, \\ 
\mc{48.144.5.ek.1}, \mc{48.144.5.el.2},
\mc{48.144.5.em.1}, 
\mc{48.144.5.en.2},
\mc{48.144.5.eo.2},\\ 
\mc{48.144.5.ep.1}, \mc{48.144.5.eq.2},
\mc{48.144.5.er.1},
\mc{48.144.5.es.1}, 
\mc{48.144.5.et.2},\\
\mc{48.144.5.eu.1},   
\mc{48.144.5.ev.2},
\mc{48.144.5.ew.2},  
\mc{48.144.5.ex.1}, 
\mc{48.144.5.ey.2}, \\
\mc{48.144.5.ez.1},
\mc{48.144.5.fa.1}, 
\mc{48.144.5.fb.2},
\mc{48.144.7.bas.1},
\mc{48.144.7.bgq.2},\\ 
\mc{48.144.11.ys.1}, 
\mc{48.144.11.yq.1}}\\

\hline

\mc{48.72.0.c.2} & \makecell[l]{
\mc{48.144.3.bb.2},
\mc{48.144.3.r.2}, 
\mc{48.144.3.t.1},
\mc{48.144.3.z.2}, \mc{48.144.5.ee.1},  \\
\mc{48.144.5.ef.2},
\mc{48.144.5.eg.2},
\mc{48.144.5.eh.1},
\mc{48.144.5.ei.1}, 
\mc{48.144.5.ej.2},\\
\mc{48.144.5.ek.2},
\mc{48.144.5.el.1}, 
\mc{48.144.5.em.2},
\mc{48.144.5.en.1},  
\mc{48.144.5.eo.1},\\
\mc{48.144.5.ep.2}, 
\mc{48.144.5.eq.1},
\mc{48.144.5.er.2},
\mc{48.144.5.es.2},
\mc{48.144.5.et.1},\\
\mc{48.144.5.eu.2},     \mc{48.144.5.ev.1}, 
\mc{48.144.5.ew.1},
\mc{48.144.5.ex.2}, 
\mc{48.144.5.ey.1}, \\
\mc{48.144.5.ez.2}, 
\mc{48.144.5.fa.2},
\mc{48.144.5.fb.1},  
\mc{48.144.7.bas.2}, \mc{48.144.7.bgq.1},\\
\mc{48.144.11.yq.2},
\mc{48.144.11.ys.2}}  \\ \hline

\mc{48.72.0.d.1} & \makecell[l]{
\mc{48.144.3.bd.2},
\mc{48.144.3.bf.1}, 
\mc{48.144.5.fi.1},
\mc{48.144.5.fw.1}, \mc{48.144.5.fz.1}, \\  \mc{48.144.5.gb.1}, \mc{48.144.5.gj.1}, 
\mc{48.144.5.gw.1},
\mc{48.144.5.gy.1}, 
\mc{48.144.5.hb.1} } \\ \hline

\mc{48.72.0.d.2} & \makecell[l]{\mc{48.144.3.bd.1}, \mc{48.144.3.bf.2}, 
\mc{48.144.5.fm.1},
\mc{48.144.5.fx.1}, \mc{48.144.5.ga.1},\\
\mc{48.144.5.gd.1},  \mc{48.144.5.gf.1},
\mc{48.144.5.gu.1}, \mc{48.144.5.gx.1}, \mc{48.144.5.ha.1} }
\end{longtable}}

These are mostly also amenable to techniques already discussed in this paper. We go through each in full detail in \texttt{Section 8.2.m}. We provide additional comments about curves with a more ad hoc approach here.

The curve \mc{48.144.3.bf.1} is a genus 3 rank 3 curve. Instead of computing its set of rational points explicitly, we check that it has commutator index 576 and compute its proper maximal subgroups that are fiber products of two respective factors, contain $-I$, and have full determinant. There are 11 such groups whose labels are in the set \{\mc{48.288.13.bhs.1}, \mc{48.288.13.bhr.1}, \mc{48.288.13.bib.1},  \mc{48.288.13.bhy.1}, \mc{48.288.21.hxr.1}, \mc{48.288.13.bft.1}, \mc{48.288.21.hxt.1}, \mc{48.288.13.bfu.1}, \mc{48.288.13.bga.1}, \mc{48.288.13.bgd.1}, \mc{48.288.21.hxl.2}\}. Each of these can be handled using methods already discussed, and this is also done in \texttt{Section 8.2.m}.

The curve \mc{48.144.3.bd.1} is a genus 3 rank 3 curve. Again, instead of computing its set of rational points explicitly, we check that it has commutator index 576 and compute its proper maximal subgroups that are fiber products of two respective factors that contain $-I$ and have full determinant. There are 11 such curves and their labels comprise the set \{\mc{48.288.21.hxh.2}, 
\mc{48.288.13.bfj.1}, \mc{48.288.13.bfk.1}, \mc{48.288.13.bgs.1}, \mc{48.288.13.bgr.1}, \mc{48.288.13.bfa.1}, \mc{48.288.13.bgi.1}, \mc{48.288.13.bfd.1}, \mc{48.288.13.bgl.1}, \mc{48.288.21.hxf.1}, \mc{48.288.21.hxp.2}\}. We similarly handle these with discussed techniques on the GitHub file.

\subsubsection{Rational points on \mc{48.72.3.a.1},\mc{48.72.3.a.2},\mc{48.72.3.p.1}, and \mc{48.72.3.p.2}}

The curve \mc{48.72.3.a.1} can be realized as \mc{16.24.0.k.2} $\times \; K_2$. We divide our analysis into three parts. We first consider the direct product of children of \mc{16.24.0.k.2} with $K_2$. Then, we consider the direct product of \mc{16.24.0.k.2} with the proper subgroups of $K_2$ in this lattice. Lastly, we consider the proper maximal subgroups that are fiber products of the two respective factors that contain $-I$, have full determinant, and could possibly contain a subgroup whose commutator index lies outside $\mathcal{I}_2$.

There are six children of \mc{16.24.0.k.2} with labels \mc{16.48.0.n.1}, \mc{16.48.1.bq.1}, \mc{16.48.0.a.1}, \mc{16.48.0.h.2}, \mc{16.48.0.m.1}, and \mc{16.48.1.bv.1}. The labels of the direct product of these with $K_2$ are \mc{48.144.8.ca.1}, \mc{48.144.7.rm.1}, \mc{48.144.8.a.1}, \mc{48.144.8.bd.1}, \mc{48.144.8.bz.2}, and \mc{48.144.7.sg.1} respectively. Let us discuss their rational points. The curve \mc{48.144.8.ca.1} covers \mc{12.36.1.bt.1} which has rank 0. It has exactly 6 points and all are cusps or CM. The curve \mc{48.144.7.rm.1} covers \mc{24.72.2.jo.1} which is bielliptic and has rank 2. It has Weierstrass model \[y^2 = 2x^6 + 2.\]  From Table \ref{tab:biellipticchabauty}, this has exactly four rational points, all of which are CM. The curve \mc{48.144.8.a.1} covers \mc{12.36.2.a.1} which has genus 2, rank 0. By using \texttt{Chabauty0}, it has exactly four rational points: two are cusps and two are CM points. The curve \mc{48.144.8.bd.1} covers \mc{12.36.1.bs.1} which we have already considered in this section. The curve \mc{48.144.8.bz.2} covers \mc{12.36.2.p.1} which has genus 2, rank 0. Using \texttt{Chabauty0}, we know that it has exactly four rational points, out of which two are cusps and two are CM points. The curve \mc{48.144.7.sg.1} covers \mc{24.36.1.gk.1} which has rank 0. It has exactly two rational points: one is a cusp and one is CM. 

For the second step, we consider the direct product of \mc{16.24.0.k.2} with the proper subgroups of $K_2$ in this lattice; the only such subgroup is $K_3$. The direct product of \mc{16.24.0.k.2} with $K_3$ has label \mc{48.144.7.ik.1}. This curve covers \mc{12.36.1.bq.1} which we have already discussed: it has no non-CM rational points.

Lastly, we consider the proper maximal subgroups that are fiber products of the two respective factors that contain $-I$, have full determinant, and could possibly contain a subgroup whose commutator index lies outside $\mathcal{I}_2$. This gives the following modular curves:

\begin{multicols}{4}
    \begin{itemize}
        \item \mc{48.144.6.bk.2}
        \item \mc{48.144.7.hu.1}
        \item \mc{48.144.7.kt.1}
        \item \mc{48.144.7.mg.1}
        \item \mc{48.144.7.mw.1}
        \item \mc{48.144.8.lu.1}
        \item \mc{48.144.9.baj.1}
        \item \mc{48.144.9.bay.1}
        \item \mc{48.144.9.zt.1}
        \item \mc{48.144.10.lt.1}
    \end{itemize}
\end{multicols}

In \texttt{Section 8.2.m}, we show that these curves all have no non-CM rational points. Although we have not determined the set of rational points of \mc{48.72.3.a.1}, we have now shown that none of its subgroups in our lattice could result in an adelic index outside $\mathcal{I}_2.$

We adopt a similar strategy as the previous subsection for \mc{48.72.3.a.2}, which can be realized as \mc{16.24.0.k.1} $\times \; K_2$. The children of \mc{16.24.0.k.1} comprise the set \{\mc{16.48.0.n.1}, \mc{16.48.0.h.1}, \mc{16.48.0.a.1}, \mc{16.48.1.bl.1}, \mc{16.48.0.m.2},  \mc{16.48.1.bg.1}\}. The labels of direct product of these with $K_2$ are \{\mc{48.144.8.ca.1}, \mc{48.144.8.bd.2}, \mc{48.144.8.a.1}, \mc{48.144.7.pu.1}, \mc{48.144.8.bz.1}, \mc{48.144.7.pa.1}\}. These cover \mc{12.36.1.bt.1}, \mc{12.36.1.bs.1}, \mc{12.36.2.a.1}, \mc{24.36.1.fp.1}, \mc{12.36.2.p.1}, and \mc{24.36.1.fo.1}, respectively. All of the latter curves have no non-CM rational points, so none of the former curves do either.

We next consider the direct product of \mc{16.24.0.k.1} with the proper subgroups of $K_2$ in this lattice. The only such subgroup is $K_3$. The direct product of \mc{16.24.0.k.1} with $K_3$ has label \mc{48.144.7.ik.2}. This covers \mc{12.36.1.bq.1} which we have already discussed. It has no non-CM rational points. 

Lastly, we consider all the proper maximal subgroups that are fiber products of the two respective factors that contain $-I$, have full determinant, and could possibly contain a subgroup whose commutator index lies outside $\mathcal{I}_2$. From this, we obtain the following modular curves:

\begin{multicols}{4}
    \begin{itemize}
        \item \mc{48.144.6.bk.1}
        \item \mc{48.144.7.hu.2}
        \item \mc{48.144.7.ks.1}
        \item \mc{48.144.7.mh.1}
        \item \mc{48.144.7.mx.1}
        \item \mc{48.144.8.lv.1}
        \item \mc{48.144.9.bai.1}
        \item \mc{48.144.9.zs.1}
        \item \mc{48.144.9.baz.1}
        \item \mc{48.144.10.ls.1}

    \end{itemize}
\end{multicols}

We show in \texttt{Section 8.2.m} that these all have no non-CM rational points. These all use already-discussed methods, with the slight exception of the curve \mc{48.144.7.mh.1}. This covers \mc{24.72.2.gm.1} which is genus 2, rank 2 bielliptic. This curve has no rational points, which can be verified using a version of the Mordell--Weil sieve implemented by Bianchi and Padurariu designed to prove that the set of rational points for such curves is empty \cite{simplesievecode}. To conclude, although we have not determined the set of rational points of \mc{48.72.3.a.2}, we have now shown that none of its subgroups in our lattice could result in an adelic index outside $\mathcal{I}_2.$

For \mc{48.72.3.p.1} (=\mc{16.24.0.n.1} $\times$ $X_{\mathrm{ns}}^{+}(3)$), it suffices to compute all the adelic indices of non-CM rational points of \mc{48.72.3.a.1} that come from subgroups in our lattice. This is because, from Tables 11 and 12 of \cite{MR4868206}, if there is a non-CM $E/\Q$ whose $2$-adic image is a subgroup of \mc{16.24.0.n.1}, then there exists an $E'/\Q$ that is 2-isogenous to $E$ whose $2$-adic image will be a subgroup of \mc{16.24.0.k.2}. From Lemma \ref{lemma:isomorphicimages}, admitting a 2-isogeny does not affect the $3$-adic image. So, if there is a non-CM $E/\Q$ that corresponds to a rational point on \mc{48.72.3.p.1}, then there exists an $E'/\Q$ that corresponds to a rational point on \mc{48.72.3.a.1}. Further, from Corollary 4 of \cite{MR4830946}, we know that isogenous elliptic curves share the same adelic index. Lastly, by a similar argument, it suffices for the curve \mc{48.72.3.p.2} to compute all the adelic indices of non-CM rational points of \mc{48.72.3.a.2} that come from subgroups in our lattice.

This concludes the proof of Theorem \ref{t:mainthm} for $p=2.$

\subsection{Rational Isogenies of Degree \texorpdfstring{$3$}{3}}\label{subsec:3isoglattice}

Let $E/\Q$ be a non-CM elliptic curve that admits a rational isogeny of degree $3$ and no rational isogeny of any larger prime degree. By Proposition \ref{P:3isog-elladic}, we know all possible $2$-adic and $3$-adic Galois images of $E$ that contain $-I$; we denote the sets of these images by $\mathcal{L}_2$ and $\mathcal{L}_3$, respectively. In particular, $\mathcal{L}_2$ consists of $15$ groups and  $\mathcal{L}_3$ consists of $14$ groups. There are two possibilities for the $5$-adic image: $\GL_2(\Z_5)$ and \mc{5.5.0.a.1}. However, we note that the modular curve \mc{15.20.1.a.1} $= X_0(3) \times$ \mc{5.5.0.a.1} is an elliptic curve of rank $0$. It has only two noncuspidal non-CM rational points. These give rise only to elliptic curves of adelic index $160$. For this reason, we will assume the $5$-adic Galois representation is surjective for the remainder of the analysis. Then, by Proposition \ref{P:3isog-elladic}, the $\ell$-adic Galois representation of $E$ is surjective for all $\ell \geq 5$.

\subsubsection{Maximal Products with Finitely Many Rational Points}
For each direct product $G = H \times K$ whose factors come from $H \in \mathcal{L}_2$ and  $K \in \mathcal{L}_3$, we consider whether $X_G$ has finitely or infinitely many rational points. Let $\mathcal{M}$ be the set of such products that are maximal with respect to the property of having finitely many rational points. The set $\mathcal{M}$ is determined using the function \texttt{DetermineMaxProdsWithFinManyPts} in the file \texttt{Approach\textunderscore Product.m} and the output can be viewed by using \texttt{IdentifyProds(MaxProdsWithFinManyPts)} in the file \texttt{Case p = 3.m}. We obtain $40$ modular curves. These are:

\begin{multicols}{4}
\begin{itemize}
\item \mc{6.24.1.a.1}
\item \mc{12.24.1.b.1}
\item \mc{12.32.1.b.1}
\item \mc{12.72.1.d.1}
\item \mc{12.72.1.f.1}
\item \mc{18.24.1.a.1}
\item \mc{18.36.2.d.1}
\item \mc{18.72.2.b.1}
\item \mc{18.72.2.b.2}
\item \mc{18.72.2.d.1}
\item \mc{18.108.2.a.1}
\item \mc{18.108.2.b.1}
\item \mc{18.108.2.c.1}
\item \mc{18.108.2.d.1}
\item \mc{18.108.2.d.2}
\item \mc{18.108.2.e.1}
\item \mc{24.24.1.bx.1}
\item \mc{24.72.1.t.1}
\item \mc{36.24.1.a.1}
\item \mc{36.72.1.b.1}
\item \mc{36.72.1.c.1}
\item \mc{36.72.2.a.1}
\item \mc{36.72.2.a.2}
\item \mc{36.72.2.b.1}
\item \mc{54.108.2.a.1}
\item \mc{72.24.1.a.1}
\item \mc{72.72.1.d.1}
\item \mc{72.72.2.a.1}
\item \mc{72.72.2.a.2}
\item \mc{72.72.2.b.1}
\item \mc{72.72.2.b.2}
\item \mc{72.72.2.c.1}
\item \mc{72.72.2.d.1}
\item \mc{72.72.3.bd.1}
\item \mc{72.72.3.bd.2}
\item \mc{72.72.3.bf.1}
\item \mc{72.72.3.d.1}
\item \mc{72.72.3.l.1}
\item \mc{72.72.3.x.1}
\item \mc{216.72.3.b.1}
\end{itemize}
\end{multicols}

As before, we will only highlight unusual cases in this article. Full computations can be found on the GitHub repository. In particular, all of the computational claims about rational points for this section can be verified in the file \texttt{Section 8.3.m}.

The elliptic curve \mc{12.32.1.b.1} is of rank $0$ and has eight rational points. Its non-CM $j$-invariants are $\{\tfrac{-3^3\cdot11^3}{2^2}, \tfrac{3^2\cdot23^3}{2^6}\}$. The adelic index for both of these $j$-invariants is $128 \in \mathcal{I}_3$.

Let $C$ denote the subscheme of $\mathbb{P}^3_{\Q}$ defined by the sequence of homogeneous polynomials that represent the model of the genus 2 curve \mc{72.72.2.d.1} produced by Zywina's code. We verify that the set of points $C(\Z/4\Z)$ is empty. Thus, there are no rational points.

Next, consider \mc{72.72.3.bd.1}, which is $\mc{8.2.0.b.1} \times \mc{9.36.0.d.1}$. Suppose there is a non-CM $E/\Q$ that corresponds to a rational point on this curve. From Table 6 of \cite{rakvi2023classification3adicgalois}, the associated isogeny-torsion graph of $E$ must be $L_3(9)$. However, from Table 18 of \cite{MR4868206}, the $2$-adic image of $E$ must be $\GL_2(\Z_2)$, which is a contradiction. Therefore, all of the curve's rational points are either cusps or CM points. By the LMFDB, there are $3$ rational cusps and no CM points. Therefore, there are three rational points on this curve. The same argument is also applicable to \mc{72.72.3.bd.2}(= $\mc{8.2.0.b.1} \times \mc{9.36.0.d.2}$), \mc{72.72.3.bf.1}(=$\mc{8.2.0.b.1} \times \mc{9.36.0.e.1}$), \mc{72.72.3.d.1}(=$\mc{8.2.0.b.1} \times \mc{9.36.0.a.1}$), \mc{72.72.3.l.1}(=$\mc{8.2.0.b.1} \times \mc{9.36.0.b.1}$), \mc{72.72.3.x.1} (=$\mc{8.2.0.b.1} \times \mc{9.36.0.c.1}$) and \mc{216.72.3.b.1}(=$\mc{8.2.0.b.1} \times \mc{27.36.0.a.1}$). Therefore, all of their rational points are either cusps or CM points.

\subsubsection{Products with Infinitely many Rational Points}
Next, we consider the direct products $G = H \times K$ with $H \in \mathcal{L}_2$ and  $K \in \mathcal{L}_3$ for which $X_G$ has infinitely many rational points. By using the function \texttt{FindDeltaE(ProdsWithInfManyPts)} in the file 
\texttt{Case p = 3.m}, we find that 
\[
\delta_E \in \{1,3\}.
\]

By taking the product of commutator indices of $2$-adic image, $3$-adic image, and entanglement factor, we obtain the following set of possible indices: 
\[ \{16, 32, 48, 96, 128, 144, 160, 288, 384, 432, 768, 864, 1296, 3888\}.\]
All of these except $432$ and $3888$ appear in $\mathcal{I}_3$. Therefore, our goal for the remainder of this section is to rule out these two as possible indices. 

There are a total of $17$ curves for which a fiber product could possibly give rise to an adelic index of $432$ or $3888$. Of these, for the following $15$ curves we need to rule out the adelic index of $432$:
\begin{multicols}{4}
\begin{itemize}
\item \mc{3.12.0.a.1}
\item \mc{9.12.0.a.1}
\item \mc{9.12.0.b.1}
\item \mc{9.36.0.a.1}
\item \mc{9.36.0.c.1}
\item \mc{9.36.0.d.1}
\item \mc{9.36.0.d.2}
\item \mc{9.36.0.e.1}
\item \mc{9.36.0.f.1}
\item \mc{9.36.0.f.2}
\item \mc{9.36.0.g.1} 
\item \mc{24.24.1.cd.1}
\item \mc{72.24.0.a.1}
\item \mc{72.24.0.b.1}
\item \mc{72.24.1.b.1}
\end{itemize}
\end{multicols}
\noindent For the remaining two curves, \mc{9.36.0.b.1} and \mc{27.36.0.a.1}, we need to rule out the adelic index of $3888$.

For groups \mc{24.24.1.cd.1}, \mc{72.24.0.a.1}, \mc{72.24.0.b.1}, and \mc{72.24.1.b.1}, we observe that there are no subgroups that are fiber products of the two respective factors that contain $-I$, have full determinant, and have commutator index outside $\mathcal{I}_3$.

To address the remaining curves, we note that the commutator index of any of these groups lies in $\mathcal{I}_3.$ We then look at the maximal subgroups of $G$ that are fiber products of the two respective factors that contain $-I$, have full determinant, and may contain a subgroup whose commutator index lies outside $\mathcal{I}_3.$ If the corresponding modular curve has finitely many rational points, we find all its rational points. Otherwise, we check that the commutator index of that subgroup lies in $\mathcal{I}_3$ and repeat the process until we encounter a modular curve whose genus is larger than $1$. Performing this loop, we obtain the following set of modular curves to consider for each group.

\begin{table}[H]
\begin{tabular}{|c|l|}
\hline
$G$ & \textbf{Curves to consider}  \\
\hline\mc{3.12.0.a.1} & \mc{6.24.1.b.1}\\ \hline
    \mc{9.12.0.a.1} & \mc{18.24.1.b.1}\\ \hline
    \mc{9.12.0.b.1} & \makecell[l]{ \mc{18.72.1.e.1}, \mc{54.72.3.d.1}, \mc{54.72.4.e.1}, \mc{54.72.5.d.1}} \\ \hline
    \mc{9.36.0.a.1} & \mc{18.72.3.j.1} \\ \hline
    \mc{9.36.0.b.1} & \mc{18.72.3.m.1} \\ \hline
    \mc{9.36.0.c.1} & \mc{18.72.3.n.1} \\ \hline
    \mc{9.36.0.d.1} & \mc{18.72.3.k.1} \\ \hline
    \mc{9.36.0.d.2} & \mc{18.72.3.k.2} \\ \hline
    \mc{9.36.0.e.1} & \mc{18.72.3.o.1} \\ \hline
    \mc{9.36.0.f.1} & \mc{18.72.2.e.1} \\ \hline
    \mc{9.36.0.f.2} & \mc{18.72.2.e.2} \\ \hline
    \mc{9.36.0.g.1} & \mc{18.72.2.f.1} \\ \hline
    \mc{27.36.0.a.1} & \mc{54.72.3.c.1} \\ \hline
\end{tabular} \caption{Modular curves to consider to rule out 432, 3888}

\end{table}

These are mostly also amenable to techniques already discussed in this paper. We go through each in full detail in \texttt{Section 8.3.m}. We provide details about the curves with ad hoc approaches here.

The curve \mc{54.72.3.d.1}, which we denote by $C$, has genus 3, rank 1. It has a canonical model in $\mathbb{P}^2$ given by \[x^3y + x^3z + 3y^3z + 3y^2z^2 + 3yz^3=0.\] We will use the argument of Section \ref{subsection:sieving}. The curve $C$ has at least three rational points. Set $P_1:=(1:0:0)$, $P_2:=(0:1:0)$ and $P_3:=(0:0:1)$. It has an automorphism $[x: y: z]  \mapsto [x: z: y]$. The quotient is a rank 1 elliptic curve. Therefore, there exists an abelian variety of rank $0$, which we denote $V$, such that the Jacobian of this curve decomposes as $E \times V$. For any rational point $P$ on this curve, the point $P-i(P)$ lies in the torsion subgroup of its Jacobian, which we denote by $J_C(\Q)_{\tors}$. It suffices to compute preimages of rational points in $J_C(\Q)_{\tors}$ under the map $a \colon C(\Q) \to J_C(\Q)_{\tors}$ given by $P \mapsto P-i(P)$. Note that $a$ is injective away from the fixed points of $i$. By using local computations, we can deduce that $J_C(\Q)_{\tors}$ is a subgroup of $\Z/3\Z \times \Z/3\Z$. Further, we observe that the classes of the divisors $D_1:=P_1-P_2$ and $D_2:=P_2-P_3$ are of order 3 and generate $J_C(\Q)_{\tors}$. One can see that $P_2-i(P_2)$ is equivalent to $D_2$ and $P_3-i(P_3)$ is equivalent to $2D_2.$ Further, by working modulo $5$, we check that there is no divisor of the form $P - i(P)$ which is equivalent to $aD_1+bD_2$ for $(a,b)$ other than $(0,1)$ and $(0,2).$ Therefore, the only rational points of $C$ are the fixed points of $i$ and $\{P_2,P_3\}$. It is easy to check that the only fixed point is $P_1$. This means that the only rational points are the known ones. Two are cusps and one is a CM point. 

The curves \mc{54.72.4.e.1} and \mc{54.72.5.d.1} have a map to some rank 0 elliptic curve. Using the code \texttt{ModCrvToEC} \cite{JacobJeremycode} we determine that they have four and three rational points, respectively. All are either cusps or CM.

The curve \mc{18.72.2.f.1} has genus 2, rank 2, and is bielliptic with Weierstrass model \[y^2= - x^6 + 3x^5 + 3x^4 - 11x^3 + 3x^2 + 3x - 1.\] We used the following bielliptic model for this curve \[y^2 = 9x^6 - 99x^4 + 27x^2 - 1.\] By Table \ref{tab:biellipticchabauty}, it has six points all of which are CM points.

This concludes the proof of Theorem \ref{t:mainthm} for $p=3,$ the final case. With this, the proof of the theorem is complete.
\appendix

\section{Example Elliptic Curves} \label{A:ExCurves}

The table below contains an example of an elliptic curve over $\Q$ for each rational isogeny degree and adelic index appearing in Theorem \ref{t:mainthm}, verifying our claim that the theorem is sharp.

{\small \begin{center} 
\begin{longtable}{|c|l|l|l|}

\hline \textbf{LMFDB Label}    & \textbf{Weierstrass Model}                         & \textbf{Isog.\ Degree} & \textbf{Adelic Index} \\ \hline
\endfirsthead
\multicolumn{4}{c}
{{\bfseries \tablename\ \thetable{} -- continued from previous page}} \\
\hline \textbf{LMFDB Label}    & \textbf{Weierstrass Model}                         & \textbf{Isog.\ Degree} & \textbf{Adelic Index}\\ \hline 
\endhead
\hline \multicolumn{4}{|r|}{\textit{Continued on next page}} \\ \hline
\endfoot

\hline
\endlastfoot

\ec{46.a1}     & $y^2+xy=x^3-x^2-170x-812$                 & 2                   & 12           \\ \hline
\ec{33.a1}     & $y^2+xy=x^3+x^2-146x+621$                 & 2                   & 48           \\ \hline
\ec{1568.a1}   & $y^2=x^3+x^2-3544x-82104$                 & 2                   & 72           \\ \hline
\ec{34.a1}     & $y^2+xy=x^3-113x-329$                     & 2                   & 96           \\ \hline
\ec{726.b1}    & $y^2+xy=x^3+x^2-475x-4187$                & 2                   & 144          \\ \hline
\ec{21.a1}     & $y^2+xy=x^3-784x-8515$                    & 2                   & 192          \\ \hline
\ec{66.c1}     & $y^2+xy=x^3-10065x-389499$                & 2                   & 288          \\ \hline
\ec{20.a1}     & $y^2=x^3+x^2-41x-116$                     & 2                   & 384          \\ \hline
\ec{126350.t1} & $y^2+xy=x^3-x^2-665127947x+6188783038661$ & 2                   & 576          \\ \hline
\ec{15.a1}     & $y^2+xy+y=x^3+x^2-2160x-39540$            & 2                   & 768          \\ \hline
\ec{14.a1}     & $y^2+xy+y=x^3-2731x-55146$                & 2                   & 864          \\ \hline
\ec{4225.e1}   & $y^2+xy=x^3-1470388x-686359233$           & 2                   & 1152         \\ \hline
\ec{17.a1}     & $y^2+xy+y=x^3-x^2-91x-310$                & 2                   & 1536         \\ \hline
\ec{44.a1}     & $y^2=x^3+x^2-77x-289$                     & 3                   & 16           \\ \hline
\ec{242.a1}    & $y^2+xy+y=x^3-6295x-197958$               & 3                   & 32           \\ \hline
\ec{300.b1}    & $y^2=x^3-x^2-1213x-15863$                 & 3                   & 48           \\ \hline
\ec{34.a1}     & $y^2+xy=x^3-113x-329$                     & 3                   & 96           \\ \hline
\ec{162.a1}    & $y^2+xy=x^3-x^2-6x+8$                     & 3                   & 128          \\ \hline
\ec{26.a1}     & $y^2+xy+y=x^3-460x-3830$                  & 3                   & 144          \\ \hline
\ec{324.b1}    & $y^2=x^3-39x+94$                          & 3                   & 160          \\ \hline
\ec{676.a1}    & $y^2=x^3+x^2-88612x-10182444$             & 3                   & 288          \\ \hline
\ec{20.a1}     & $y^2=x^3+x^2-41x-116$                     & 3                   & 384          \\ \hline
\ec{162.b1}    & $y^2+xy=x^3-x^2-1077x+13877$              & 3                   & 768          \\ \hline
\ec{14.a1}     & $y^2+xy+y=x^3-2731x-55146$                & 3                   & 864          \\ \hline
\ec{19.a1}     & $y^2+y=x^3+x^2-769x-8470$                 & 3                   & 1296         \\ \hline
\ec{38.b1}     & $y^2+xy+y=x^3+x^2-70x-279$                & 5                   & 48           \\ \hline
\ec{14450.h1}  & $y^2+xy=x^3+x^2-2622825x-1639310375$      & 5                   & 192          \\ \hline
\ec{66.c1}     & $y^2+xy=x^3-10065x-389499$                & 5                   & 288          \\ \hline
\ec{50.a1}     & $y^2+xy+y=x^3-126x-552$                   & 5                   & 384          \\ \hline
\ec{338.b1}    & $y^2+xy=x^3+x^2-54421x+4945517$           & 5                   & 576          \\ \hline
\ec{43264.f1}  & $y^2=x^3-x^2-6231x-187253$                & 5                   & 864          \\ \hline
\ec{11.a1}     & $y^2+y=x^3-x^2-7820x-263580$              & 5                   & 1200         \\ \hline
\ec{26.b1}     & $y^2+xy+y=x^3-x^2-213x-1257$              & 7                   & 96           \\ \hline
\ec{507.b1}    & $y^2+xy+y=x^3+x^2-75x-1422$               & 7                   & 192          \\ \hline
\ec{1922.c1}   & $y^2+xy+y=x^3-x^2-76332x+8136267$         & 7                   & 576          \\ \hline
\ec{162.b1}    & $y^2+xy=x^3-x^2-1077x+13877$              & 7                   & 768          \\ \hline
\ec{121.a1}    & $y^2+xy+y=x^3+x^2-305x+7888$              & 11                  & 480          \\ \hline
\ec{147.b1}    & $y^2+y=x^3-x^2-912x+10919$                & 13                  & 336          \\ \hline
\ec{14450.b1}  & $y^2+xy+y=x^3-190891x-36002922$           & 17                  & 576          \\ \hline
\ec{1225.b1}   & $y^2+xy+y=x^3+x^2-208083x-36621194$       & 37                  & 2736       \\
\end{longtable}
\end{center} }

\bibliographystyle{amsalpha}
\bibliography{References}

\end{document}